\documentclass[11pt, a4paper,reqno]{amsart}

\usepackage[utf8]{inputenc}
\usepackage{amsmath}
\usepackage[dvipsnames]{xcolor}
\usepackage{amssymb,esint}
\usepackage{amsthm}
\usepackage{dsfont}
\usepackage[shortlabels]{enumitem}
\usepackage[british]{babel}
\usepackage{hyperref}
\usepackage[]{mdframed}
\usepackage{mathrsfs}
\usepackage{mathtools}

\calclayout
\hypersetup{colorlinks=true,linkcolor=blue,citecolor=blue,urlcolor=blue}
\theoremstyle{definition}
\newtheorem{theorem}{Theorem}[section]
\newtheorem{prop}[theorem]{Proposition}
\newtheorem{definition}[theorem]{Definition}
\newtheorem{remark}[theorem]{Remark}
\newtheorem{lemma}[theorem]{Lemma}
\newtheorem{corollary}[theorem]{Corollary}
\numberwithin{equation}{section}

\newcommand{\abs}[1]{\left|#1\right|}
\newcommand{\norm}[1]{\left\|#1\right\|}
\newcommand{\R}{\mathbb R}
\newcommand{\N}{\mathbb N}
\newcommand{\A}{\mathcal A}
\newcommand{\dd}{\, \mathrm d}
\newcommand{\ddd}{\mathrm d}
\newcommand{\id}{\mathrm{Id}}
\newcommand{\cF}{\mathcal F}
\newcommand{\D}{\mathcal D}
\newcommand{\m}{\mathbf m}
\newcommand{\homd}{\mathsf{Q}}
\newcommand{\cc}{\mathfrak c}
\newcommand{\cW}{\mathcal W}
\newcommand{\E}{\mathcal E}
\newcommand{\frG}{\mathfrak G}
\DeclareMathOperator{\supp}{supp}
\DeclareMathOperator{\PV}{p.\!v.}
\newcommand{\one}{\mathds{1}}
\renewcommand{\L}{\operatorname{L}} 
\newcommand{\C}{\operatorname{C}} 
\renewcommand{\H}{\operatorname{H}} 
\DeclareRobustCommand{\Wdot}{\dot{\operatorname{W}}\protect{\vphantom{W}}}
\DeclareRobustCommand{\Bdot}{\dot{\operatorname{B}}\protect{\vphantom{B}}}
\DeclareRobustCommand{\Fdot}{\dot{\operatorname{F}}\protect{\vphantom{F}}}
\newcommand{\Dv}{\mathfrak D_v^{\sigma}}
\newcommand{\Dfr}{\mathfrak D^{\sigma}}
\newcommand{\Dvs}{\mathfrak D_v^{\sigma,*}}
\newcommand{\Dfrs}{\mathfrak D^{\sigma,*}}
\newcommand{\comp}{\subset\!\subset}

\begin{document}
\allowdisplaybreaks
\date{\today}

\title[On the kinetic $p$-Laplace equation with nonlocal diffusion]{On the kinetic $p$-Laplace equation\\ with nonlocal diffusion}

\author{Lukas Niebel}
\address[Lukas Niebel]{Institut f\"ur Analysis und Numerik, Universit\"at M\"unster\\
    Orl\'eans-Ring 10, 48149 M\"unster, Germany.}
\email{lukas.niebel@uni-muenster.de}
\date{\today}




\begin{abstract}
    \emergencystretch=1em
    We study two nonlocal versions of the kinetic $p$-Laplace equation: a Gagliardo-type model defined through
    differences and a Bessel-type model defined via Fourier multiplication. Using critical kinetic
    trajectories, we derive representation formulas adapted to the kinetic transport-diffusion geometry and
    establish homogeneous and scale-invariant kinetic Gagliardo--Nirenberg inequalities for nonlocal diffusion,
    which yield gain-of-integrability estimates for weak solutions to the kinetic $p$-Laplace equations
    with nonlocal diffusion.
\end{abstract}

\maketitle
\enlargethispage{3pt}

\section{Introduction}
Let $d\in\N$, $0<\sigma<1$, and $p \in (1,\infty)$. In this paper we study two natural formulations of the
kinetic $p$-Laplace equation with nonlocal diffusion. We begin with the Gagliardo-type model
\begin{equation}
    \label{eq:kinetic-frac-p}
    (\partial_t+v\cdot\nabla_x)f + (-\Delta_v)_p^{\sigma} f =0
    \qquad\text{in }\R^{1+2d},
\end{equation}
where
\begin{equation}
    \label{eq:def-frac-plap-intro}
    [(-\Delta_v)_p^{\sigma} f](t,x,v)\!
    := \PV\int_{\R^d} \!
    \frac{\abs{f(t,x,v)-f(t,x,w)}^{p-2}\bigl(f(t,x,v)-f(t,x,w)\bigr)}{\abs{v-w}^{d+\sigma p}}
    \dd w.\!\!
\end{equation}
The natural energy space for weak solutions is the fractional Gagliardo--Sobolev space. We refer to
\cite{MR2944369} for further details. This equation is invariant under the scaling
\[
    f_\lambda(t,x,v)= f(\lambda^{\sigma p}t, \lambda^{\sigma p+1}x,\lambda v).
\]
For $h\in\R^d$ we set
\[
    [\delta_h f](t,x,v):=f(t,x,v+h)-f(t,x,v),
    \qquad
    \Dv f(t,x,v,h):=\frac{\delta_h f(t,x,v)}{\abs{h}^{\sigma}},
\]
so that, in the weak divergence-form sense, the nonlocal $p$-Laplacian is represented by
\[
    (-\Delta_v)_p^{\sigma}f
    =
    -\frac12 \Dvs\Bigl(\abs{\Dv f}^{p-2}\Dv f\Bigr)
\]
for suitably smooth functions.
Here $\Dvs$ denotes the adjoint nonlocal divergence in the velocity variable.
Consequently, the weak divergence-form version of \eqref{eq:kinetic-frac-p} is
\begin{equation}
    \label{eq:linear-structure-intro}
    (\partial_t+v\cdot\nabla_x)f=\Dvs S,
    \qquad
    S:=\frac12 \abs{\Dv f}^{p-2}\Dv f.
\end{equation}
When the principal-value integral in \eqref{eq:def-frac-plap-intro} exists
pointwise, this weak formulation agrees with the pointwise one.

For the Gagliardo model, our first goal is to derive, by averaging along critical kinetic trajectories, a
representation formula for functions satisfying \eqref{eq:linear-structure-intro}. The resulting control is
expressed only in terms of the source term $S$ and a forcing term of order $\Dv f$. Our second goal is to
convert this representation formula into a global gain-of-integrability estimate.

Compared with kinetic equations with local diffusion, the main new difficulty is that the forcing term along
trajectories is no longer controlled by $\nabla_v f$. Instead, we exploit cancellation in the
velocity-increment variable to recover the pointwise density of the fractional seminorm. The source term is
treated through a nonlocal integration-by-parts argument, which produces a kernel with the expected fractional
tail in the increment variable.

In the present paper, we establish estimates for sufficiently smooth functions and do not address the density
argument needed to obtain embeddings for the corresponding kinetic Sobolev spaces. We refer to
\cite{auscher2026kineticsobolevspaces,MR4807233,dietert2025criticaltrajectorieskineticgeometry} for further
discussion.

Our first main result is the following kinetic Gagliardo--Nirenberg inequality for Gagliardo-type nonlocal diffusion. We write $\dd
    \eta(h):=\frac{\dd h}{\abs{h}^d}$.
\begin{theorem}
    \label{thm:int:mainGN-differencemp}
    Let $0<\sigma<1$ and
    \[
        p\in \Bigl(2-\frac{2\sigma}{2d+d\sigma+2\sigma},\,2+\frac{2}{d}\Bigr).
    \]
    Define
    \begin{equation*}
        q:=\frac{2p(d\sigma+d+\sigma)}{d(p\sigma+2)}.
    \end{equation*}
    There exists $C = C(d,\sigma,p)>0$ such that the following holds.

    Assume that $f\in \L^p(\R^{1+2d})$ and $S\in \L^{\frac{p}{p-1}}(\R^{1+2d}\times\R^d,\ddd \eta\otimes\ddd (t,x,v))$ are sufficiently smooth, with
    \[
        \Dv f\in \L^p(\R^{1+2d}\times\R^d,\ddd \eta\otimes\ddd (t,x,v)),
    \]
    and that
    \[
        (\partial_t+v\cdot\nabla_x)f=\Dvs S
    \]
    in the sense of distributions on $\R^{1+2d}$. Then
    \begin{equation*}
        \|f\|_{\vphantom{\L^{\frac{p}{p-1}}(\ddd \eta\otimes\ddd (t,x,v))}\L^q(\R^{1+2d})}
        \le
        C\,
        \|\Dv f\|_{\vphantom{\L^{\frac{p}{p-1}}(\ddd \eta\otimes\ddd (t,x,v))}
                \L^p(\ddd \eta\otimes\ddd (t,x,v))}^{\frac{d\sigma+2d+\sigma}{2(d\sigma+d + \sigma)}}
        \|S\|_{\L^{\frac{p}{p-1}}(\ddd \eta\otimes\ddd (t,x,v))}^{\frac{\sigma(d+1)}{2(d\sigma+d+\sigma)}}.
    \end{equation*}
\end{theorem}
\bigskip
The estimate of this theorem is homogeneous and scale invariant. In particular,
the assumption $f \in \L^p$ is qualitative only.
\smallskip
As a direct consequence of Theorem \ref{thm:int:mainGN-differencemp}, weak solutions to \eqref{eq:kinetic-frac-p} gain integrability in all variables
$(t,x,v)$ above the energy level.

\begin{corollary}    \label{cor:gagliardo-p-laplace}
    Let $\sigma,p$ be as in Theorem \ref{thm:int:mainGN-differencemp} and $f\in \L^p(\R^{1+2d})$
    be sufficiently smooth with
    $\Dv f\in \L^p(\R^{1+2d}\times\R^d,\ddd \eta\otimes\ddd (t,x,v))$, and suppose that $f$ is a weak solution to \eqref{eq:linear-structure-intro}.
    Then $f\in \L^q(\R^{1+2d})$ with $q$ as in Theorem \ref{thm:int:mainGN-differencemp}, and
    \[
        \|f\|_{\L^q(\R^{1+2d})}
        \lesssim
        \|\Dv f\|_{\L^p(\ddd \eta\otimes\ddd (t,x,v))}^{\Gamma},
        \qquad
        \Gamma:=\frac{p\sigma(d+1)+2d}{2(d\sigma+d+\sigma)}.
    \]
\end{corollary}

\bigskip

We next consider a second kinetic $p$-Laplace equation with nonlocal diffusion, adapted to the setting of homogeneous Bessel
potential spaces. In this case, the nonlocal derivative in the velocity variable is defined by the Fourier
multiplier
\[
    \widehat{D_v^\sigma g}(\xi)=\abs{\xi}^\sigma \widehat g(\xi),
    \qquad \xi\in\R^d.
\]
The corresponding inverse, the Riesz potential $I_v^\sigma=(-\Delta_v)^{-\sigma/2}$, is defined on mean-zero
Schwartz functions by
\[
    \widehat{I_v^\sigma h}(\xi)=\abs{\xi}^{-\sigma}\widehat h(\xi).
\]
The associated kinetic $p$-Laplace equation is
\begin{equation}
    \label{eq:kinetic-nonlocal-p}
    (\partial_t+v\cdot\nabla_x)f+D_v^\sigma\bigl(\abs{D_v^\sigma f}^{p-2}D_v^\sigma f\bigr)=0.
\end{equation}
In general, the Gagliardo and Bessel formulations are not equivalent. They do, however, share the same scaling
invariance
\[
    f_\lambda(t,x,v)= f(\lambda^{\sigma p}t, \lambda^{\sigma p+1}x,\lambda v).
\]
For this model we prove an analogous representation formula for functions $f$ satisfying
\begin{equation}
    \label{eq:kinetic-nonlocal-source}
    (\partial_t+v\cdot\nabla_x)f=D_v^\sigma S.
\end{equation}
The resulting estimate is controlled by $D_v^\sigma f$ and the source term $S$,
and is our second main result.

\begin{theorem}
    \label{thm:int:mainGN}
    Let $0<\sigma<1$ and let
    \[
        p\in \Bigl(2-\frac{2\sigma}{2d+d\sigma+2\sigma},\,2+\frac{2}{d}\Bigr),
        \qquad q:=\frac{2p(d\sigma+d+\sigma)}{d(p\sigma+2)}.
    \]
    There exists $C=C(d,p,\sigma)>0$ such that the following holds.
    If $f\in \L^p(\R^{1+2d})$ is sufficiently smooth, with $D_v^\sigma f\in \L^p(\R^{1+2d})$, and satisfies
    \[
        (\partial_t+v\cdot\nabla_x)f=D_v^\sigma S
    \]
    in the sense of distributions for some sufficiently regular $S\in \L^{\frac{p}{p-1}}(\R^{1+2d})$, then
    \[
        \|f\|_{\vphantom{\L^{\frac{p}{p-1}}}\L^{q}(\R^{1+2d})}
        \le C
        \|D_v^\sigma f\|_{\vphantom{\L^{\frac{p}{p-1}}}\L^p(\R^{1+2d})}
            ^{\frac{d\sigma+2d+\sigma}{2(d\sigma+d+\sigma)}}
        \|S\|_{\L^{\frac{p}{p-1}}(\R^{1+2d})}^{\frac{\sigma(d+1)}{2(d\sigma+d+\sigma)}}.
    \]
\end{theorem}

\bigskip

This estimate is homogeneous and scale invariant. In particular, the assumption $f \in \L^p$ is qualitative
only.

\begin{remark}
    \label{rem:gagliardo-not-from-bessel}
    It is natural to ask whether Theorem~\ref{thm:int:mainGN-differencemp}
    can be deduced from Theorem~\ref{thm:int:mainGN}. This seems to be true only in
    the Hilbertian setting $p=2$.

    Indeed, suppose that $f$ satisfies the assumptions of
    Theorem~\ref{thm:int:mainGN-differencemp}, and write
    $p'=p/(p-1)$. In order to apply Theorem~\ref{thm:int:mainGN}, one would
    have to rewrite
    \[
        (\partial_t+v\cdot\nabla_x)f=\Dvs S
    \]
    in the Bessel form
    \[
        (\partial_t+v\cdot\nabla_x)f=D_v^\sigma S_B,
        \qquad
        S_B:=I_v^\sigma\Dvs S .
    \]
    Thus one would need the two estimates
    \begin{equation}
        \label{eq:needed-gagliardo-to-bessel-f}
        \|D_v^\sigma f\|_{\L^p(\R^{1+2d})}
        \lesssim
        \|\Dv f\|_{\L^p(\ddd\eta\otimes\ddd(t,x,v))}
    \end{equation}
    and
    \begin{equation}
        \label{eq:needed-gagliardo-to-bessel-source}
        \|I_v^\sigma\Dvs S\|_{\L^{p'}(\R^{1+2d})}
        \lesssim
        \|S\|_{\L^{p'}(\ddd\eta\otimes\ddd(t,x,v))}.
    \end{equation}
    These two bounds are compatible only when $p=2$. In the velocity
    variable, the seminorm $\|\Dv f\|_{\L^p(\ddd\eta\otimes\ddd v)}$
    corresponds to the homogeneous Besov norm
    $\Bdot^\sigma_{p,p}$, whereas $\|D_v^\sigma f\|_{\L^p_v}$
    corresponds to the homogeneous Triebel--Lizorkin norm
    $\Fdot^\sigma_{p,2}$. Hence
    \eqref{eq:needed-gagliardo-to-bessel-f} is the embedding
    $\Bdot^\sigma_{p,p} \hookrightarrow \Fdot^\sigma_{p,2}$.
    By the homogeneous Besov--Lizorkin--Triebel embeddings \cite[(1.2), p.~95]{MR454618},
    \[
        \Bdot^\sigma_{p,\min\{p,2\}}
        \hookrightarrow
        \Fdot^\sigma_{p,2}
        \hookrightarrow
        \Bdot^\sigma_{p,\max\{p,2\}},
    \]
    this holds in the required direction for $p\le 2$, but fails in general
    for $p>2$.

    On the other hand, by duality, \eqref{eq:needed-gagliardo-to-bessel-source}
    is equivalent to
    \[
        \|\Dv I_v^\sigma \varphi\|_{\L^p(\ddd\eta\otimes\ddd(t,x,v))}
        \lesssim
        \|\varphi\|_{\L^p(\R^{1+2d})} .
    \]
    In the velocity variable this is the opposite embedding,
    $\Fdot^\sigma_{p,2} \hookrightarrow \Bdot^\sigma_{p,p}$, which holds
    in the required direction for $p\ge 2$, but fails in general
    for $p<2$. Consequently, the two estimates needed to pass from the
    Gagliardo formulation to the Bessel formulation are simultaneously
    available only for $p=2$.

    For $p=2$, Plancherel gives
    \[
        \|D_v^\sigma f\|_{\L^2(\R^{1+2d})}
        \sim
        \|\Dv f\|_{\L^2(\ddd\eta\otimes\ddd(t,x,v))},
    \]
    and the corresponding dual estimate for $I_v^\sigma\Dvs S$. Thus
    Theorem~\ref{thm:int:mainGN} gives the estimate of
    Theorem~\ref{thm:int:mainGN-differencemp}. For $p\neq2$, however, it seems that
    Theorem~\ref{thm:int:mainGN-differencemp} is not a direct consequence of
    Theorem~\ref{thm:int:mainGN}.
\end{remark}

We can deduce the following gain of integrability for weak solutions to \eqref{eq:kinetic-nonlocal-p}.

\begin{corollary}
    \label{cor:nonlocal-p-laplace}
    Let $\sigma,p$ be as in Theorem \ref{thm:int:mainGN} and $f\in \L^p(\R^{1+2d})$ be sufficiently smooth,
    with $D_v^\sigma f\in \L^p(\R^{1+2d})$, and suppose that $f$ is a weak solution to
    \eqref{eq:kinetic-nonlocal-p}. Then $f\in \L^q(\R^{1+2d})$ with $q$ as in Theorem \ref{thm:int:mainGN}, and
    \[
        \|f\|_{\L^q(\R^{1+2d})}
        \lesssim
        \|D_v^\sigma f\|_{\L^p(\R^{1+2d})}^{\Gamma},
        \qquad
        \Gamma:=\frac{p\sigma(d+1)+2d}{2(d\sigma+d+\sigma)}.
    \]
\end{corollary}

In fact, the estimates of Theorems \ref{thm:int:mainGN-differencemp} and \ref{thm:int:mainGN} apply more
generally to nonlinear kinetic equations with nonlocal diffusion and $p$-growth, and in particular can be used
in the presence of rough diffusion coefficients. The same proofs apply verbatim to nonnegative weak sub- and
supersolutions; we refer to \cite{dietert2025criticaltrajectorieskineticgeometry,dnz_kinpLaplace_2026} for the
technical details.
One may also allow $S$ and the velocity derivative of $f$ to lie in Lebesgue spaces with different exponents;
compare \cite{dnz_kinpLaplace_2026}.
The estimates may be localised, at the cost of suitable tail terms due to their nonlocal nature.
Other source terms can be included in both theorems; compare \cite{dnz_kinpLaplace_2026,n_kinetictor_2026}.

The estimates obtained in this work open up the study of regularity theory for nonlinear kinetic equations
with nonlocal diffusion of $p$-growth. Questions of interest include, for example, local boundedness, a priori
H\"older continuity and weak Harnack inequalities.

Let us also comment on the range of $p$ in Theorems \ref{thm:int:mainGN-differencemp} and
\ref{thm:int:mainGN}. This is the nonlocal analogue of the restriction in \cite{dnz_kinpLaplace_2026}. It
seems possible that one may be able to improve the lower bound when studying localised estimates. However, in
the global setting we conjecture this range to be optimal.

\subsection*{Literature.}
Theorems \ref{thm:int:mainGN-differencemp} and \ref{thm:int:mainGN} are available in the literature in
the case $p = 2$. Here, the Gagliardo and Bessel formulations agree up to normalisation. The sharp gain of
integrability is explained first in \cite{MR4049224}; see also \cite{MR4688651}. An alternative proof, which
yields only a subcritical gain of integrability, is given in \cite{MR4039522} based on velocity-averaging. In
functional-analytic form, the endpoint embedding is contained in the kinetic Sobolev space theory in
\cite{auscher2025weaksolutionskolmogorovfokkerplanckequations,auscher2026kineticsobolevspaces}. To the best of
our knowledge, the proof given in the present manuscript is the first to obtain the critical gain of
integrability while not relying on the fundamental solution of the fractional Kolmogorov equation.

Nonlinear kinetic equations with local
diffusion of $p$-growth were studied in \cite{dnz_kinpLaplace_2026}. In that paper, kinetic trajectories were adapted to the nonlinear
structure and then used to prove kinetic Gagliardo--Nirenberg inequalities, which in turn led to local
boundedness of subsolutions to the kinetic $p$-Laplace equation with local diffusion via De~Giorgi iteration.

Kinetic trajectories were introduced in \cite{MR4875497}, while critical kinetic trajectories and a systematic
account of their use were developed in \cite{dietert2025criticaltrajectorieskineticgeometry}. They provide a
flexible framework for estimates in kinetic spaces, where regularity is measured through
$\partial_t+v\cdot\nabla_x$ and velocity derivatives. Applications include Poincar\'e inequalities
\cite{MR4875497,anceschi2025poincareinequalityquantitativegiorgi}, gain of integrability for the Kolmogorov
equation \cite{dietert2025criticaltrajectorieskineticgeometry}, and transfer of regularity
\cite{n_kinetictor_2026}. Moreover, they can be used to prove the Harnack inequality for the Kolmogorov
equation with rough diffusion coefficients
\cite{dietert2025criticaltrajectorieskineticgeometry,dietert2025nashsgboundkolmogorov,%
    anceschi2025poincareinequalityquantitativegiorgi}. We also note that
\cite{anceschi2025poincareinequalityquantitativegiorgi} used kinetic trajectories in a nonlocal setting to
prove a kinetic Poincar\'e inequality; our approach is partly inspired by their treatment of the fractional
derivatives. Kinetic trajectories were adapted to nonlinear kinetic equations in \cite{dnz_kinpLaplace_2026}.

A more classical approach to proving inequalities in kinetic Sobolev spaces is based on the fundamental
solution; see \cite{MR2068847,auscher2025weaksolutionskolmogorovfokkerplanckequations,%
    auscher2026kineticsobolevspaces,MR4049224,MR4336467,MR4350284,MR4898687}. We also mention the
influential work \cite{MR1949176}.

In the linear case $p=2$ the Kolmogorov equation with nonlocal diffusion and rough jump kernel was first
studied in \cite{MR4049224}, where a weak Harnack inequality and a gain of integrability were obtained.
Alternative proofs were later given in \cite{MR4688651,MR4039522}. It was shown in \cite{MR4828346} that the
Harnack inequality fails in general.

For nonlinear kinetic equations with rough diffusion and $p=2$ growth, a De Giorgi--Nash--Moser theory in the
local case was developed in \cite{MR4468371} based on the linear theory of \cite{MR3923847}. Boundedness
properties for nonlinear kinetic equations with rough nonlocal diffusion and $p=2$ growth were studied in
\cite{MR4950599}. There, the gain-of-integrability problem for nonlinear equations with $p$-growth was
identified as a difficult open problem. The present paper resolves that question.

Outside the kinetic setting, the standard fractional $p$-Laplacian already has a substantial elliptic and
parabolic theory. On the elliptic side, nonlocal Harnack inequalities and local regularity for fractional
$p$-minimisers were established in \cite{MR3237774,MR3542614}, complemented by global H\"older regularity up
to the boundary in \cite{MR3593528}; see also \cite{MR4030247} for the equivalence of the main notions of
solution. In the parabolic setting, existence, uniqueness and asymptotic behaviour for the fractional
$p$-Laplacian flow were studied in \cite{MR3491533}, while local boundedness and H\"older continuity were
obtained in \cite{MR4204564,MR4679981}. We also mention the two recent works
\cite{MR4964333,giovagnoli2025c1alpharegularityfractionalpharmonic}. For nonlinear nonlocal $p$-Laplace type
operators built on the space $\H^{\sigma,p}$, which are closer in spirit to the Bessel-type model studied in this
manuscript, see \cite{MR3403430,MR3714833}.

\subsection*{Notation}
Throughout, let $d\in \N$ be the spatial dimension. The subscripts $t$, $x$, and
$v$ indicate, respectively, the first variable, the next $d$ variables, and the
final $d$ variables. We use these subscripts for functions, gradients $\nabla$,
Lebesgue spaces $\L^p$, and Sobolev spaces $\H^s$, with the intended meaning
being clear from the context. We write $\L^p$ for the usual Lebesgue spaces and
$\L^{p,\infty}$ for the corresponding weak Lebesgue spaces. The relations
$\lesssim,\sim$ are understood up to a universal multiplicative
constant. In proofs, $C$ denotes a universal constant whose value may vary from
one line to the next. For $1<p<\infty$, we write $p'=p/(p-1)$.

\subsection*{Acknowledgements} The author thanks Helge Dietert and Rico Zacher for fruitful discussions. Lukas
Niebel is funded by the Deutsche Forschungsgemeinschaft (DFG, German Research Foundation) under Germany's
Excellence Strategy EXC 2044/2--390685587, Mathematics M\"unster: Dynamics--Geometry--Structure.

\section{On the Gagliardo-type nonlocal \texorpdfstring{$p$}{p}-Laplace operator}
\label{sec:prelim}

We first introduce the kinetic $p$-Laplace equation with Gagliardo-type
nonlocal diffusion and its weak divergence-form formulation.
We set
\[
    [\delta_h f](t,x,v):=f(t,x,v+h)-f(t,x,v),
    \qquad
    \Dv f(t,x,v,h):=\frac{\delta_h f(t,x,v)}{\abs{h}^{\sigma}},
\]
and work with the measure $\dd \eta(h):=\frac{\dd h}{\abs{h}^d}$. Then the
natural seminorm is
\[
    \begin{aligned}
        [f]_{\L^p_{t,x}\Wdot_v^{\sigma,p}(\R^{1+2d})}
         & :=
        \Biggl(\int_{\R^{1+2d}}\int_{\R^d}
        \frac{\abs{\delta_h f(t,x,v)}^p}{\abs{h}^{d+\sigma p}}
        \dd h\dd(t,x,v)\Biggr)^{1/p} \\
         & =
        \|\Dv f\|_{\L^p(\R^{1+2d}\times\R^d,\ddd \eta\otimes\ddd(t,x,v))}.
    \end{aligned}
\]
When no confusion can arise, we suppress the $(t,x)$-dependence.
The adjoint divergence introduced below is the corresponding nonlocal
divergence operator in the spirit of the nonlocal vector calculus developed in
\cite{MR2728700,MR3023366}; see also the recent divergence-theorem formulation in
\cite{MR4711580}. A closely related Gagliardo-type divergence-form setting for
nonlocal $p$-Laplacian equations is used in \cite{MR4942309}.

\begin{definition}
    Let $F\colon \R^d\times\R^d\to\R$ be a measurable function for which the
    integral below converges absolutely.
    We define
    \begin{equation*}
        \Dvs F(v)
        :=
        \int_{\R^d}
        \frac{F(v,h)-F(v-h,h)}{\abs{h}^{d+\sigma}}\dd h.
    \end{equation*}
\end{definition}

\begin{lemma}
    \label{lem:adjoint}
    Let $\varphi\in \C_c^{\infty}(\R^d)$ and let
    $F\in \L^1_{\mathrm{loc}}(\R^d\times\R^d,
        \frac{\ddd h}{\abs{h}^d}\otimes \ddd v)$ be such that all integrals below
    are finite. Then
    \begin{equation*}
        \int_{\R^d} \varphi(v)\,\Dvs F(v)\dd v
        =
        -\int_{\R^d}\int_{\R^d}
        \Dv\varphi(v,h)\,F(v,h)\,\frac{\dd h}{\abs{h}^d}\dd v.
    \end{equation*}
\end{lemma}

\begin{proof}
    This follows by Fubini and the change of variables $v\mapsto v+h$.
\end{proof}

For suitable $f$, and following the usual divergence-form realisation of
nonlocal $p$-Laplacian-type operators (see \cite{MR4942309}), we use the operator
\begin{equation*}
    \mathcal L_{p,\sigma}f
    :=
    -\frac12 \Dvs\bigl(\abs{\Dv f}^{p-2}\Dv f\bigr)
\end{equation*}
defined in the weak sense.
More precisely, if
\[
    \int_Q\int_{\R^d}\int_{\R^d}
    \abs{\Dv f(t,x,v,h)}^p
    \frac{\dd h}{\abs{h}^d}\dd v \dd (t,x)
    <\infty
    \qquad\text{for every }Q\comp\R^{1+d}_{t,x},
\]
then $\mathcal L_{p,\sigma}f$ is a well-defined distribution determined by
\begin{equation*}
    \bigl\langle \mathcal L_{p,\sigma}f,\varphi\bigr\rangle
    :=
    \frac12
    \int_{\R^{1+2d}} \int_{\R^d}
    \abs{\Dv f}^{p-2}\Dv f\,\Dv\varphi
    \frac{\dd h}{\abs{h}^d}\dd(t,x,v),
\end{equation*}
for every $\varphi\in\C_c^\infty(\R^{1+2d})$.

To compare with the pointwise notation, whenever the principal value below
exists, we write
\begin{equation*}
    [(-\Delta_v)_p^{\sigma} f](t,x,v)
    :=
    \PV\int_{\R^d}
    \frac{\abs{f(t,x,v)-f(t,x,w)}^{p-2}
        \bigl(f(t,x,v)-f(t,x,w)\bigr)}
    {\abs{v-w}^{d+\sigma p}}
    \dd w.
\end{equation*}
This is the standard principal-value fractional $p$-Laplacian, up to sign and
normalisation conventions; see for instance
\cite{MR3542614,MR3593528,MR4030247,MR4303657}.

\begin{remark}
    The identification of the weak and principal-value formulations for smooth
    functions is delicate in the singular range $1<p<2$, especially at critical
    points. For $\C^2$ functions, a standard sufficient condition is
    $p>2/(2-\sigma)$; otherwise, one works
    away from points where $\nabla_v f=0$; compare
    \cite[Section~3.2, Lemma~3.6]{MR4030247}. For suitably smooth functions the
    weak and principal-value formulations coincide. See also
    \cite[Section~1]{MR4303657} for the standard
    principal-value definition and its relation with the Gagliardo energy.
\end{remark}

\section{Kinetic trajectories}
\label{sec:kintraj}
We now define critical kinetic trajectories and derive representation formulas adapted to nonlocal diffusion.
All $2\times 2$ matrices act on $\R^{2d}$ through their tensor product with $\id_d$.

\subsection{Critical kinetic trajectories}
Let $0< \alpha<\beta<\infty$.
We consider the two forcing functions
\[
    g_1(r)=r^{\beta}\sin(\log r),
    \qquad
    g_2(r)=r^{\beta}\cos(\log r),
    \qquad r > 0.
\]
For $(t,x,v)\in\R^{1+2d}$, $r \ge 0$, $m_0\in\R\setminus\{0\}$ and $(m_1,m_2)\in\R^{2d}$ we define
\[
    \gamma^{\m}(r;(t,x,v))
    :=(\gamma_t^{\m},\gamma_x^{\m},\gamma_v^{\m})(r;(t,x,v))
\]
by
\begin{equation}
    \label{eq:gammam-nl}
    \gamma^{\m}(r;(t,x,v))
    :=
    \begin{pmatrix}
        t+m_0r^{\alpha} \\
        \displaystyle \E_{m_0}(r)\binom{x}{v} + \D_{m_0}\cW(r)\D_{m_0}^{-1}\binom{m_1}{m_2}
    \end{pmatrix},
\end{equation}
where, for $\delta\in\R$ and $r\ge0$,
\begin{equation}
    \label{eq:defDandE-nl}
    \cW(r):=
    \begin{pmatrix}
        g_1(r)                                   & g_2(r)                                   \\
        \dfrac{\dot g_1(r)}{\alpha r^{\alpha-1}} & \dfrac{\dot g_2(r)}{\alpha r^{\alpha-1}}
    \end{pmatrix},
    \qquad
    \D_{\delta}:=
    \begin{pmatrix}
        \delta & 0 \\
        0      & 1
    \end{pmatrix},
    \qquad
    \E_{\delta}(r):=
    \begin{pmatrix}
        1 & \delta r^{\alpha} \\
        0 & 1
    \end{pmatrix}.
\end{equation}
We emphasise that $\gamma^{\m}(0;(t,x,v))$ is to be understood as the limit $r \to 0$, which exists as $\alpha
    < \beta$.
We abbreviate
\[
    \A_{m_0}(r):=\D_{m_0}\cW(r)\D_{m_0}^{-1}
    \qquad\text{and}\qquad
    \homd:=(2\beta-\alpha)d+\alpha.
\]
The following properties are easy to verify; compare
\cite{dietert2025criticaltrajectorieskineticgeometry,dnz_kinpLaplace_2026,n_kinetictor_2026}.

\begin{enumerate}[label=(M\arabic*),ref=(M\arabic*),itemsep=0.15cm]
    \item \label{item:kinetic-trajectory} $\gamma^{\m}$ is a kinetic trajectory, that is
          \[
              \dot\gamma_x^{\m}(r)=\dot\gamma_t^{\m}(r)\,\gamma_v^{\m}(r),
              \qquad \dot\gamma_t^{\m}(r)=\alpha m_0r^{\alpha-1}.
          \]
    \item \label{item:kinetic-jacobian} There exists a universal constant $\cc = \cc_{\alpha,d}\neq0$ such
          that
          \[
              \det\bigl(\A_{m_0}(r)\bigr)=\cc\,r^{(2\beta-\alpha)d}
              \qquad\text{for all }r>0.
          \]
    \item \label{item:kinetic-inverse-bounds} For $i=1,2$ and all $r>0$,
          \[
              \abs{\bigl(\A_{m_0}(r)^{-1}\bigr)_{i;1}}
              \lesssim (1+\abs{m_0}^{-1})r^{-\beta},
              \qquad
              \abs{\bigl(\A_{m_0}(r)^{-1}\bigr)_{i;2}}
              \lesssim (1+\abs{m_0})r^{\alpha-\beta}.
          \]
    \item \label{item:kinetic-trajectory-bounds} For all $r>0$,
          \[
              \begin{cases}
                  \abs{\dot\gamma_v^{\m}(r)}
                  \lesssim \Bigl(\dfrac{\abs{m_1}}{\abs{m_0}}+\abs{m_2}\Bigr)r^{\beta-\alpha-1},
                  \\[0.2cm]
                  \abs{\gamma_v^{\m}(r)-v}
                  \lesssim \Bigl(\dfrac{\abs{m_1}}{\abs{m_0}}+\abs{m_2}\Bigr)r^{\beta-\alpha},
                  \\[0.2cm]
                  \abs{\gamma_x^{\m}(r)-x-m_0vr^{\alpha}}
                  \lesssim (\abs{m_1}+\abs{m_0}\abs{m_2})r^{\beta}.
              \end{cases}
          \]
\end{enumerate}
We also introduce the forcing matrix
\begin{equation*}
    \cF_{m_0}(r)
    :=
    \begin{pmatrix}
        \dfrac1{m_0}\partial_r\dfrac{\dot g_1(r)}{\alpha r^{\alpha-1}}
         &
        \partial_r\dfrac{\dot g_2(r)}{\alpha r^{\alpha-1}}
    \end{pmatrix}\in\R^{d\times 2d},
\end{equation*}
so that
\[
    \dot\gamma_v^{\m}(r;(t,x,v))
    =\cF_{m_0}(r)\binom{m_1}{m_2}.
\]

\subsection{Kinetic mollification}
We recall the kinetic translation group
\[
    (t,x,v)\circ (s,y,w)=(t+s,x+y+sv,v+w),
    \qquad
    (t,x,v)^{-1}=(-t,-x+tv,-v).
\]
For convenience, we also record
\[
    (t,x,v)^{-1}\circ(s,y,w)=(s-t,y-x-(s-t)v,w-v).
\]
Let $0 \le \psi\in \C_c^{\infty}(\R^{1+2d})$ satisfy
\[
    \supp\psi\subset (-2,-1)\times B_1(0)\times B_1(0),
    \qquad
    \int_{\R^{1+2d}}\psi\dd \m=1.
\]
The support in negative times is adapted to proving estimates for nonnegative weak subsolutions. For the
estimates proved here, it matters only that the time projection of the support is bounded away from zero. To
treat supersolutions, one would replace $(-2,-1)$ by $(1,2)$.

Let $N \in \N$. For a kernel $J\colon \R^{1+2d}\to\R^N$ we define the kinetic convolution
\begin{equation*}
    [T_Jf](t,x,v)
    :=\int_{\R^{1+2d}} f(\m) \cdot J\bigl((t,x,v)^{-1}\circ\m\bigr)\dd \m
\end{equation*}
on functions $f \colon \R^{1+2d} \to \R^N$.
If $J\colon \R^{1+2d}\times\R^d\to\R$, we define the extended convolution
\begin{equation}
    \label{eq:defTJ-extended}
    [T_J^{\eta}F](t,x,v)
    :=\int_{\R^{1+2d}}\int_{\R^d}
    F(\m,h)J\bigl((t,x,v)^{-1}\circ \m,h\bigr)
    \dd \eta(h)\dd \m.
\end{equation}

For $\tau>0$ we define the kinetic mollifier
\begin{equation}
    \label{eq:kernelKtau}
    K_\tau(s,y,w):=\abs{\cc}^{-1}\tau^{-\homd}
    \psi\Bigl(\frac{s}{\tau^\alpha},\A_{s/\tau^\alpha}(\tau)^{-1}\binom{y}{w}\Bigr).
\end{equation}
A change of variables gives
\begin{equation}
    \label{eq:kinmollifier}
    [T_{K_\tau}f](t,x,v)=\int_{\R^{1+2d}} f\bigl(\gamma^{\m}(\tau;(t,x,v))\bigr)\psi(\m)\dd \m.
\end{equation}

Constants will depend on $\alpha$ and the choice of $\psi$, but we do not investigate the dependence.

\subsection{Representation formulas}

In this section we derive representation formulas for functions satisfying either
\eqref{eq:linear-structure-intro} or \eqref{eq:kinetic-nonlocal-source}, where the regularity of $f$ is
controlled either in terms of $\Dv f$ or $D_v^\sigma f$.

\subsubsection{Representation I}

Here we work with differences, which makes the representation formulas slightly more involved.

For an $h$-dependent kernel $J=J(s,y,w,h)$ and an
$h$-dependent function $F=F(t,x,v,h)$, we recall the extended convolution
notation from \eqref{eq:defTJ-extended}, namely
\[
    [T_J^\eta F](t,x,v)
    :=
    \int_{\R^{1+2d}}\int_{\R^d}
    F(\m,h)\,
    J\bigl((t,x,v)^{-1}\circ \m,h\bigr)
    \dd \eta(h)\dd \m.
\]

We also fix the constant $c_{d,\sigma}^{\mathrm{Gag}}\,>0$ by
\[
    -\Dfrs_w \Dfr_w\phi
    =
    c_{d,\sigma}^{\mathrm{Gag}}\,(-\Delta_w)^\sigma\phi
    \qquad\text{for all }\phi\in \C_c^\infty(\R^d).
\]
For $V\in \C_c^\infty(\R^{1+2d};\R^d)$ we define its fractional
antidivergence in the $w$-variable by
\[
    \mathcal R_{\sigma,w}V
    :=
    \frac1{c_{d,\sigma}^{\mathrm{Gag}}}\,
    \Dfr_w(-\Delta_w)^{-\sigma}\bigl(\nabla_w\!\cdot V\bigr).
\]
Here and below, for mean-zero $g\in \C_c^\infty(\R^d)$, we define
$(-\Delta_w)^{-\sigma}g$ distributionally by the homogeneous Fourier
multiplier
\[
    \widehat{(-\Delta_w)^{-\sigma}g}(\xi)
    =
    \abs{\xi}^{-2\sigma}\widehat g(\xi),
    \qquad \xi\neq0.
\]
Since $\widehat g(0)=0$, and in fact $\widehat g(\xi)=O(\abs{\xi})$ near
$\xi=0$, the right-hand side is locally integrable for $0<\sigma<1$.
When $2\sigma<d$, this agrees with the usual Riesz potential of order
$2\sigma$. In the borderline or supercritical cases $2\sigma\ge d$, we
use only this Fourier-multiplier definition, not the subcritical
Riesz-kernel representation. For
$g=\nabla_w\!\cdot V(s,y,\cdot)$, the mean-zero condition holds because
$V$ is compactly supported in $w$.
By construction,
\[
    -\Dfrs_w(\mathcal R_{\sigma,w}V)
    =
    \nabla_w\!\cdot V.
\]

For $r>0$ and $(s,y,w)\in\R^{1+2d}$, define
\[
    H_r(s,y,w):=\frac{\alpha s}{r}K_r(s,y,w),
\]
\[
    B_r(s,y,w)
    :=
    \cF_{s/r^\alpha}(r)\,
    \A_{s/r^\alpha}(r)^{-1}\binom{y}{w},
    \qquad
    \Theta_r(s,y,w):=B_r(s,y,w)\,K_r(s,y,w).
\]
Next define the $h$-dependent kernels
\[
    \frG_{r,\sigma}(s,y,w,h)
    :=
    [\Dfr_w H_r](s,y,w,h)
    =
    \frac{H_r(s,y,w+h)-H_r(s,y,w)}{\abs{h}^\sigma},
\]
and
\[
    \frG^v_{r,\sigma}(s,y,w,h)
    :=
    [\mathcal R_{\sigma,w}\Theta_r]((s,y,w),h)
    =
    \frac1{c_{d,\sigma}^{\mathrm{Gag}}}\,
    \Dfr_w(-\Delta_w)^{-\sigma}\bigl(\nabla_w\!\cdot \Theta_r\bigr)((s,y,w),h).
\]
The latter is well-defined as $\Theta_r(s,y,w)$ is compactly supported in $w$.

\begin{prop}
    \label{prop:representation-formula-nonlocal}
    Assume that $f\colon \R^{1+2d} \to \R$ and $S\colon \R^{1+2d} \times \R^d \to \R$ are sufficiently smooth
    and satisfy
    \[
        (\partial_t+v\cdot\nabla_x)f=\Dvs S
        \qquad\text{in }\R^{1+2d}.
    \]
    Then for every $\tau>0$ and every $(t,x,v)\in\R^{1+2d}$,
    \begin{equation}
        \label{eq:representation-formula-nonlocal}
        f(t,x,v)-[T_{K_\tau}f](t,x,v)
        =
        \int_0^\tau [T_{\frG_{r,\sigma}}^\eta S](t,x,v)\dd r
        +
        \int_0^\tau [T_{\frG^v_{r,\sigma}}^\eta(\Dv f)](t,x,v)\dd r.
    \end{equation}
\end{prop}

\begin{proof}
    Fix $(t,x,v)\in\R^{1+2d}$. By \eqref{eq:kinmollifier},
    \[
        [T_{K_r}f](t,x,v)
        =
        \int_{\R^{1+2d}} f(\gamma^\m(r;(t,x,v)))\psi(\m)\dd \m.
    \]
    Since $[T_{K_r}f](t,x,v)\to f(t,x,v)$ as $r\downarrow0$, we have
    \[
        f(t,x,v)-[T_{K_\tau}f](t,x,v)
        =
        -\int_0^\tau \frac{\dd}{\dd r}[T_{K_r}f](t,x,v)\dd r.
    \]
    Differentiating under the integral sign and using~\ref{item:kinetic-trajectory} gives
    \begin{align*}
         & \frac{\dd}{\dd r}f(\gamma^\m(r;(t,x,v))) \\
         & =
        \alpha m_0r^{\alpha-1}
        \bigl[(\partial_t+v\cdot\nabla_x)f\bigr](\gamma^\m(r;(t,x,v)))
        +
        \dot\gamma_v^\m(r;(t,x,v))
        \cdot
        [\nabla_v f](\gamma^\m(r;(t,x,v))).
    \end{align*}
    Using the structural equation at the point $\gamma^\m(r;(t,x,v))$, we obtain
    \begin{align*}
         & \frac{\dd}{\dd r}f(\gamma^\m(r;(t,x,v))) \\
         & =
        \alpha m_0r^{\alpha-1}
            [\Dvs S](\gamma^\m(r;(t,x,v)))
        +
        \dot\gamma_v^\m(r;(t,x,v))
        \cdot
        [\nabla_v f](\gamma^\m(r;(t,x,v))).
    \end{align*}
    Hence
    \[
        \frac{\ddd}{\ddd r}[T_{K_r}f](t,x,v)
        =
        I_r^S(t,x,v)+I_r^f(t,x,v),
    \]
    where
    \[
        I_r^S(t,x,v)
        :=
        \int_{\R^{1+2d}}
        \alpha m_0r^{\alpha-1}
            [\Dvs S](\gamma^\m(r;(t,x,v)))
        \psi(\m)\dd \m,
    \]
    and
    \[
        I_r^f(t,x,v)
        :=
        \int_{\R^{1+2d}}
        \dot\gamma_v^\m(r;(t,x,v))
        \cdot
        [\nabla_v f](\gamma^\m(r;(t,x,v)))
        \psi(\m)\dd \m.
    \]

    We now change variables from the trajectory parameter
    $\m=(m_0,m_1,m_2)$ to the relative kinetic variables $(s,y,w)$ defined by
    \[
        (s,y,w)
        =
        (t,x,v)^{-1}\circ\gamma^\m(r;(t,x,v)).
    \]
    Equivalently,
    \[
        (t,x,v)\circ(s,y,w)
        =
        \gamma^\m(r;(t,x,v)).
    \]
    By \eqref{eq:gammam-nl}, \eqref{eq:defDandE-nl} and~\ref{item:kinetic-jacobian}, this change of variables
    gives
    \[
        s=m_0r^\alpha,
        \qquad
        \binom{y}{w}
        =
        \A_{m_0}(r)\binom{m_1}{m_2},
        \qquad
        \dd  (s,y,w)
        =
        \abs{\cc}r^{\homd}\dd\m.
    \]
    Therefore,
    \[
        m_0=\frac{s}{r^\alpha},
        \qquad
        \binom{m_1}{m_2}
        =
        \A_{s/r^\alpha}(r)^{-1}\binom{y}{w},
    \]
    and, by the definition of $K_r$,
    \[
        \psi(\m)\dd \m
        =
        K_r(s,y,w) \dd(s,y,w).
    \]

    Hence the source term becomes
    \[
        I_r^S(t,x,v)
        =
        \int_{\R^{1+2d}}
        [\Dvs S]\bigl((t,x,v)\circ(s,y,w)\bigr)
        H_r(s,y,w)
        \dd(s,y,w) ,
    \]
    because
    \[
        \alpha m_0r^{\alpha-1}K_r(s,y,w)
        =
        \frac{\alpha s}{r}K_r(s,y,w)
        =
        H_r(s,y,w).
    \]
    Moreover, since
    \[
        \dot\gamma_v^\m(r;(t,x,v))
        =
        \cF_{m_0}(r)\binom{m_1}{m_2}
        =
        \cF_{s/r^\alpha}(r)
        \A_{s/r^\alpha}(r)^{-1}\binom{y}{w},
    \]
    the term containing $\nabla_v f$ becomes
    \[
        I_r^f(t,x,v)
        =
        \int_{\R^{1+2d}}
        [\nabla_v f]\bigl((t,x,v)\circ(s,y,w)\bigr)
        \cdot
        \Theta_r(s,y,w)
        \dd(s,y,w).
    \]

    We first treat the term containing $S$. In the relative variables
    $(s,y,w)$, changing the velocity variable of
    $(t,x,v)\circ(s,y,w)$ is the same as changing the variable $w$.
    Thus, using the adjointness of $\Dfr_w$ and $\Dfrs_w$,
    \begin{align*}
        -I_r^S(t,x,v)
         & =
        -\int_{\R^{1+2d}}
        [\Dvs S]\bigl((t,x,v)\circ(s,y,w)\bigr)
        H_r(s,y,w)
        \dd(s,y,w)           \\
         & =
        \int_{\R^{1+2d}}\int_{\R^d}
        S\bigl((t,x,v)\circ(s,y,w),h\bigr)
        [\Dfr_w H_r](s,y,w,h)
        \dd\eta(h)\dd(s,y,w) \\
         & =
        \int_{\R^{1+2d}}\int_{\R^d}
        S\bigl((t,x,v)\circ(s,y,w),h\bigr)
        \frG_{r,\sigma}(s,y,w,h)
        \dd\eta(h)\dd(s,y,w) \\
         & =
        [T_{\frG_{r,\sigma}}^{\eta}S](t,x,v).
    \end{align*}

    We next treat the term containing $\nabla_v f$. Since
    \[
        \nabla_w\bigl[f\bigl((t,x,v)\circ(s,y,w)\bigr)\bigr]
        =
        [\nabla_v f]\bigl((t,x,v)\circ(s,y,w)\bigr),
    \]
    integration by parts in the $w$-variable gives
    \begin{align*}
        -I_r^f(t,x,v)
         & =
        -\int_{\R^{1+2d}}
        \nabla_w\bigl[f\bigl((t,x,v)\circ(s,y,w)\bigr)\bigr]
        \cdot
        \Theta_r(s,y,w)
        \dd(s,y,w) \\
         & =
        \int_{\R^{1+2d}}
        f\bigl((t,x,v)\circ(s,y,w)\bigr)
        \nabla_w\!\cdot\Theta_r(s,y,w)
        \dd(s,y,w).
    \end{align*}
    By definition of $\frG^v_{r,\sigma}$,
    \[
        -\Dfrs_w\frG^v_{r,\sigma}
        =
        \nabla_w\!\cdot\Theta_r.
    \]
    Therefore,
    \begin{align*}
         & -I_r^f(t,x,v) \\
         & =
        \int_{\R^{1+2d}}
        f\bigl((t,x,v)\circ(s,y,w)\bigr)
        \bigl(-\Dfrs_w\frG^v_{r,\sigma}\bigr)(s,y,w)
        \dd(s,y,w)       \\
         & =
        \int_{\R^{1+2d}}\int_{\R^d}
        \Dfr_w\bigl[f\bigl((t,x,v)\circ(s,y,w)\bigr)\bigr](s,y,w,h)
        \frG^v_{r,\sigma}(s,y,w,h)
        \dd\eta(h)\dd(s,y,w).
    \end{align*}
    Since changing $w$ in $(t,x,v)\circ(s,y,w)$ is the same as changing
    the velocity variable in the physical point, we have
    \[
        \Dfr_w\bigl[f\bigl((t,x,v)\circ(s,y,w)\bigr)\bigr](s,y,w,h)
        =
        [\Dv f]\bigl((t,x,v)\circ(s,y,w),h\bigr).
    \]
    Hence
    \begin{align*}
        -I_r^f(t,x,v)
         & =
        \int_{\R^{1+2d}}\int_{\R^d}
        [\Dv f]\bigl((t,x,v)\circ(s,y,w),h\bigr)
        \frG^v_{r,\sigma}(s,y,w,h)
        \dd\eta(h)\dd(s,y,w) \\
         & =
        [T_{\frG^v_{r,\sigma}}^{\eta}(\Dv f)](t,x,v).
    \end{align*}

    Combining the identities for $I_r^S(t,x,v)$ and $I_r^f(t,x,v)$ and
    integrating in $r\in(0,\tau)$ proves
    \eqref{eq:representation-formula-nonlocal}.
\end{proof}

\subsubsection{Representation II}
We consider the structural equation
\begin{equation*}
    (\partial_t+v\cdot\nabla_x)f=D_v^\sigma S_0,
\end{equation*}
with $S_0\in \C_c^\infty(\R^{1+2d})$, and we assume control of $D_v^\sigma f$. In this case, the
representation formula is very similar to the local case; compare
\cite{dietert2025criticaltrajectorieskineticgeometry,dnz_kinpLaplace_2026}.
The main difference is that the kernels are nonlocal in the sense that they are no longer compactly supported.

For the trajectory forcing we introduce the local vector-valued kernel
\begin{equation}
    \label{eq:def-forcing-kernel}
    \vec K_r(s,y,w)
    :=-\abs{\cc}^{-1}r^{-\homd}
    \psi\Bigl(\frac{s}{r^\alpha},\A_{s/r^\alpha}(r)^{-1}\binom{y}{w}\Bigr)
    \cF_{s/r^\alpha}(r)\A_{s/r^\alpha}(r)^{-1}\binom{y}{w}.
\end{equation}
We also define
\begin{equation}
    \label{eq:def-Lr}
    L_r(s,y,w):=-\nabla_w\cdot \vec K_r(s,y,w).
\end{equation}
Since $\vec K_r$ is compactly supported in the $w$-variable, one has
\begin{equation}
    \label{eq:mean-zero-Lr}
    \int_{\R^d}L_r(s,y,w)\dd w=0
    \qquad\text{for all }(s,y)\in\R^{1+d}.
\end{equation}
Hence the Riesz potential
\begin{equation*}
    G^v_{r,\sigma}(s,y,\cdot):=I_w^\sigma L_r(s,y,\cdot)
\end{equation*}
is well defined.
Moreover, we set
\begin{equation}
    \label{eq:def-bessel-G-kernel}
    G_{r,\sigma}(s,y,w)
    :=-\alpha \abs{\cc}^{-1}r^{\alpha-1-\homd}\frac{s}{r^\alpha}
    D_w^\sigma\Biggl[\psi\Bigl(\frac{s}{r^\alpha},\A_{s/r^\alpha}(r)^{-1}\binom{y}{w}\Bigr)\Biggr].
\end{equation}

\begin{prop}
    \label{prop:representation}
    Let $f,S_0 \colon \R^{1+2d} \to \R$ be sufficiently smooth and satisfy
    \begin{equation}
        (\partial_t+v\cdot\nabla_x)f=D_v^\sigma S_0.
    \end{equation}
    Then for every $\tau>0$,
    \begin{equation}
        \label{eq:representation-main}
        f(t,x,v)-[T_{K_\tau}f](t,x,v)
        =
        \Bigl[T_{\int_0^\tau G_{r,\sigma} \dd r}(S_0)+T_{\int_0^\tau G^v_{r,\sigma} \dd r}(D_v^\sigma
            f)\Bigr](t,x,v).
    \end{equation}
\end{prop}

\begin{proof}
    By \eqref{eq:kinmollifier} and the fundamental theorem of calculus,
    \begin{align*}
         & f(t,x,v)-[T_{K_\tau}f](t,x,v)                                                         \\
         & =-\int_{\R^{1+2d}}\int_0^\tau \frac{\dd}{\dd r}f\bigl(\gamma^{\m}(r;(t,x,v))\bigr)\dd
        r\,\psi(\m)\dd \m                                                                        \\
         & =-\int_{\R^{1+2d}}\int_0^\tau \alpha
        m_0r^{\alpha-1}\bigl[(\partial_t+v\cdot\nabla_x)f\bigr]\bigl(\gamma^{\m}(r;(t,x,v))\bigr)\dd
        r\,\psi(\m)\dd \m                                                                        \\
         & \hspace{2cm}-\int_{\R^{1+2d}}\int_0^\tau \dot\gamma_v^{\m}(r)\cdot [\nabla_v
            f]\bigl(\gamma^{\m}(r;(t,x,v))\bigr)\dd r\,\psi(\m)\dd \m.
    \end{align*}

    We treat the transport term first.
    Fix $r>0$.
    With the change of variables $\widetilde \m=\gamma^{\m}(r;(t,x,v))$ and
    property~\ref{item:kinetic-jacobian},
    \begin{align*}
         & -\int_{\R^{1+2d}} \alpha m_0 r^{\alpha-1} [D_v^\sigma
        S_0]\bigl(\gamma^{\m}(r;(t,x,v))\bigr)\psi(\m)\dd \m                                                \\
         & = -\int_{\R^{1+2d}} \alpha r^{\alpha-1}\frac{\widetilde m_0-t}{r^\alpha}\abs{\cc}^{-1}r^{-\homd}
        [D_{v}^\sigma S_0](\widetilde \m)                                                                   \\
         & \qquad \qquad\qquad\qquad \cdot \psi\Bigl(\frac{\widetilde m_0-t}{r^\alpha},
        \A_{\frac{\widetilde m_0-t}{r^\alpha}}(r)^{-1}
        \Bigl(\binom{\widetilde m_1}{\widetilde m_2}-\E_{\frac{\widetilde
                        m_0-t}{r^\alpha}}(r)\binom{x}{v}\Bigr)
        \Bigr)\dd \widetilde \m.
    \end{align*}
    Since $D_v^\sigma$ is self-adjoint and translation invariant in the velocity variable,
    \begin{align*}
         & -\int_{\R^{1+2d}} \alpha m_0 r^{\alpha-1} [D_v^\sigma
        S_0]\bigl(\gamma^{\m}(r;(t,x,v))\bigr)\psi(\m)\dd \m                     \\
         & =\int_{\R^{1+2d}} S_0(\widetilde \m)
        G_{r,\sigma}\bigl((t,x,v)^{-1}\circ \widetilde \m\bigr)\dd \widetilde \m \\
         & = [T_{G_{r,\sigma}}(S_0)](t,x,v).
    \end{align*}

    \smallskip
    Next, we study the forcing term.
    Again by the same change of variables,
    \[
        -\int_{\R^{1+2d}} \dot\gamma_v^{\m}(r)\cdot [\nabla_v f]\bigl(\gamma^{\m}(r;(t,x,v))\bigr)\psi(\m)\dd
        \m
        =[T_{\vec K_r}(\nabla_v f)](t,x,v).
    \]
    Integrating by parts in the velocity variable gives
    \[
        [T_{\vec K_r}(\nabla_v f)](t,x,v)=[T_{L_r}(f)](t,x,v),
    \]
    where $L_r$ is defined in \eqref{eq:def-Lr}.
    By \eqref{eq:mean-zero-Lr}, the function $w\mapsto L_r(s,y,w)$ has zero average for every $(s,y)$.
    Hence $G^v_{r,\sigma}=I_w^\sigma L_r$ is well defined and satisfies $D_w^\sigma G^v_{r,\sigma}=L_r$.
    Using again self-adjointness and translation invariance of $D_v^\sigma$ we obtain
    \begin{align*}
        [T_{L_r}(f)](t,x,v)
         & =\int_{\R^{1+2d}} f(\widetilde \m)L_r\bigl((t,x,v)^{-1}\circ \widetilde \m\bigr)\dd \widetilde \m
        \\
         & =\int_{\R^{1+2d}} [D_{v}^\sigma f](\widetilde \m)
        G^v_{r,\sigma}\bigl((t,x,v)^{-1}\circ \widetilde \m\bigr)\dd \widetilde \m                           \\
         & =[T_{G^v_{r,\sigma}}(D_v^\sigma f)](t,x,v).
    \end{align*}
    Combining the two contributions and integrating in $r\in(0,\tau)$ proves \eqref{eq:representation-main}.
\end{proof}

\begin{remark}
    In the case of subsolutions and supersolutions the representation formulas are replaced by the corresponding
    one-sided inequalities. In order to obtain estimates on $f$ one needs its nonnegativity as an additional
    input.
\end{remark}

\subsection{Young inequalities}
The following Young inequalities, including their weak versions, are identical to the classical ones because
the kinetic group structure
\[
    (t,x,v) \mapsto (t,x,v)\circ(s,y,w)
\]
is measure preserving; see \cite[Theorem 1.2.12 and Theorem 1.4.25]{MR2445437}.

\begin{lemma}
    \label{lem:young-nl}
    Let $1\le p,q,\theta\le\infty$ with $1/q+1=1/\theta+1/p$.
    Then for every $J\in \L^\theta(\R^{1+2d})$,
    \[
        \|T_JF\|_{\L^q(\R^{1+2d})}\le \|J\|_{\L^\theta(\R^{1+2d})}\,\|F\|_{\L^p(\R^{1+2d})}, \qquad F \in \L^p(\R^{1+2d}).
    \]
    If $1<p,q,\theta<\infty$ and $J\in \L^{\theta,\infty}(\R^{1+2d})$, then
    \[
        \|T_JF\|_{\L^q(\R^{1+2d})}\lesssim \|J\|_{\L^{\theta,\infty}(\R^{1+2d})}\,\|F\|_{\L^p(\R^{1+2d})}, \qquad F \in \L^p(\R^{1+2d}).
    \]
\end{lemma}

\subsection{Kernel bounds}
We now estimate the kernels in the representation formula.
The compactly supported kernels $K_r$ and $L_r$ behave as in the local theory.
The new point is the effect of the fractional derivative and its inverse in the $w$-variable.
On a first reading, readers may wish to skip this subsection and identify which kernel estimates are
needed in the proofs of the main results, i.e.\ Theorems~\ref{thm:int:mainGN-differencemp}
and~\ref{thm:int:mainGN}.

We begin with bounds for the compactly supported kernels.

\begin{lemma}
    \label{lem:compact-kernels}
    For every $r>0$ the kernels $K_r$ and $L_r$ satisfy
    \[
        \supp K_r\cup\supp L_r
        \subset
        \Bigl\{(s,y,w): \abs{s}\sim r^\alpha,\ \abs{y}\lesssim r^\beta,\ \abs{w}\lesssim
        r^{\beta-\alpha}\Bigr\},
    \]
    and
    \[
        \abs{K_r(s,y,w)}\lesssim r^{-\homd},
        \qquad
        \abs{L_r(s,y,w)}\lesssim r^{-1-\homd}.
    \]
    Moreover,
    \[
        \int_{\R^d}L_r(s,y,w)\dd w=0
        \qquad\text{for all }(s,y)\in\R^{1+d}.
    \]
\end{lemma}

\begin{proof}
    The support statement follows from property~\ref{item:kinetic-trajectory-bounds}: if one of the kernels is
    non-zero, then
    \[
        \frac{s}{r^\alpha}\in(-2,-1),
        \qquad
        \A_{s/r^\alpha}(r)^{-1}\binom{y}{w}\in B_1(0)\times B_1(0),
    \]
    which yields $\abs{s}\sim r^\alpha$, $\abs{w}\lesssim r^{\beta-\alpha}$ and $\abs{y}\lesssim r^\beta$.
    The bound for $K_r$ is immediate from \eqref{eq:kernelKtau}.
    For $L_r$ we differentiate \eqref{eq:def-forcing-kernel} in $w$.
    The matrix $\cF_{s/r^\alpha}(r)$ is $O(r^{\beta-\alpha-1})$ by~\ref{item:kinetic-trajectory-bounds}, while
    differentiation of the cut-off factor contributes $r^{\alpha-\beta}$ by~\ref{item:kinetic-inverse-bounds}.
    Both terms are therefore of size $r^{-1-\homd}$.
\end{proof}

\begin{lemma}
    \label{lem:kernel-Kr}
    Let $\theta\in[1,\infty]$.
    Then, uniformly in $r>0$,
    \[
        \|K_r\|_{\L^\theta}\lesssim r^{-\homd(1-\frac1\theta)}.
    \]
\end{lemma}

\begin{proof}
    This is a direct consequence of Lemma~\ref{lem:compact-kernels}.
\end{proof}

In the next technical lemma, we prove that the rescaled profiles used below are uniformly
bounded in $\C^N$ norms. We introduce the dilation
\[
    \delta_r(t,x,v)
    :=
    (r^\alpha t ,r^\beta x,r^{\beta-\alpha}v).
\]

\begin{lemma}
    \label{lem:rescaled-profile-bounds}
    Let $\mathcal I\comp \R\setminus\{0\}$ be a compact interval containing the projection of $\supp\psi$
    onto
    the time variable. For every $N\in\N$, the families
    \begin{align*}
        \widetilde H_r(\bar s, \bar y, \bar w)
         & :=
        r^{1+\homd-\alpha}H_r(\delta_r(\bar s, \bar y, \bar w)), \\
        \widetilde \Theta_r(\bar s, \bar y, \bar w)
         & :=
        r^{1+\homd-(\beta-\alpha)}\Theta_r(\delta_r(\bar s, \bar y, \bar w))
    \end{align*}
    are uniformly bounded in $\C_c^N(\R^{1+2d})$ and $\C_c^N(\R^{1+2d};\R^d)$,
    respectively. Their supports are contained in one fixed compact subset
    of $\R^{1+2d}$, independent of $r$.

    Moreover, for $\lambda\in\mathcal I$, $\bar y\in\R^d$, define
    \[
        \Phi_{r,\lambda,\bar y}(\eta)
        :=
        \psi\Bigl(
        \lambda,
        \A_\lambda(r)^{-1}
        \binom{r^\beta\bar y}{r^{\beta-\alpha}\eta}
        \Bigr),
        \qquad \eta\in\R^d,
    \]
    and
    \[
        \Psi_{r,\lambda,\bar y}(\eta)
        :=
        r^{1+\homd}
        L_r(r^\alpha\lambda,r^\beta\bar y,r^{\beta-\alpha}\eta).
    \]
    Then $\{\Phi_{r,\lambda,\bar y}\}$ is uniformly bounded in
    $\C_c^N(\R^d)$, and $\{\Psi_{r,\lambda,\bar y}\}$ is uniformly bounded
    in $\C_c^N(\R^d)$. The supports are contained in one fixed ball in
    $\R^d$. Finally,
    \[
        \int_{\R^d}\Psi_{r,\lambda,\bar y}(\eta)\dd\eta=0.
    \]
\end{lemma}

\begin{proof}
    Write $\vartheta=\log r$, and set
    \[
        a_1(\vartheta)
        :=
        \frac{1}{\alpha}\bigl(\beta\sin\vartheta+\cos\vartheta\bigr),
        \qquad
        a_2(\vartheta)
        :=
        \frac{1}{\alpha}\bigl(\beta\cos\vartheta-\sin\vartheta\bigr).
    \]
    With
    \[
        P_r:=
        \begin{pmatrix}
            r^\beta & 0                \\
            0       & r^{\beta-\alpha}
        \end{pmatrix},
    \]
    the matrix $\A_\lambda(r)$ factorises as
    \[
        \A_\lambda(r)
        =
        P_r\mathscr A_\lambda(\vartheta),
    \]
    where
    \[
        \mathscr A_\lambda(\vartheta)
        :=
        \begin{pmatrix}
            \sin\vartheta
             &
            \lambda\cos\vartheta
            \\[0.1cm]
            \lambda^{-1}a_1(\vartheta)
             &
            a_2(\vartheta)
        \end{pmatrix}.
    \]
    For each one-dimensional block,
    \[
        \det \mathscr A_\lambda(\vartheta)
        =
        \sin\vartheta\,a_2(\vartheta)
        -
        \cos\vartheta\,a_1(\vartheta)
        =
        -\frac1\alpha.
    \]
    Hence $\mathscr A_\lambda(\vartheta)$ is invertible uniformly for
    $\lambda\in\mathcal I$ and $\vartheta\in\R$. Since the coefficients
    are smooth in $\lambda$ and periodic in $\vartheta$, all
    $\lambda$- and $\vartheta$-derivatives of
    $\mathscr A_\lambda(\vartheta)^{-1}$ are uniformly bounded on
    $\mathcal I\times\R$.

    Similarly, from the definition of $\cF_\lambda(r)$,
    \[
        \cF_\lambda(r)
        =
        r^{\beta-\alpha-1}\mathscr F_\lambda(\vartheta),
    \]
    where
    \[
        \mathscr F_\lambda(\vartheta)
        :=
        \begin{pmatrix}
            \lambda^{-1}\bigl((\beta-\alpha)a_1(\vartheta)+a_1'(\vartheta)\bigr)
             &
            \bigl((\beta-\alpha)a_2(\vartheta)+a_2'(\vartheta)\bigr)
        \end{pmatrix}.
    \]
    Again, all derivatives of $\mathscr F_\lambda$ are uniformly bounded
    for $\lambda\in\mathcal I$.

    Define
    \[
        Z_{r,\lambda}(\bar y,\bar w)
        :=
        \A_\lambda(r)^{-1}
        \binom{r^\beta\bar y}{r^{\beta-\alpha}\bar w}
        =
        \mathscr A_\lambda(\vartheta)^{-1}
        \binom{\bar y}{\bar w},
    \]
    and
    \[
        V_{r,\lambda}(\bar y,\bar w)
        :=
        r^{1+\alpha-\beta}
        \cF_\lambda(r)\A_\lambda(r)^{-1}
        \binom{r^\beta\bar y}{r^{\beta-\alpha}\bar w}
        =
        \mathscr F_\lambda(\vartheta)
        \mathscr A_\lambda(\vartheta)^{-1}
        \binom{\bar y}{\bar w}.
    \]
    On bounded subsets of the $(\bar y,\bar w)$-variables, the maps
    $Z_{r,\lambda}$ and $V_{r,\lambda}$, together with all derivatives
    of any fixed order, are uniformly bounded.

    We now write the rescaled kernels explicitly. By the definition of
    $H_r$ and $K_r$,
    \[
        \widetilde H_r(\bar s,\bar y,\bar w)
        =
        \alpha\abs{\cc}^{-1}\bar s\,
        \psi\bigl(\bar s,Z_{r,\bar s}(\bar y,\bar w)\bigr).
    \]
    Likewise, since $\Theta_r=B_rK_r$,
    \[
        \widetilde \Theta_r(\bar s,\bar y,\bar w)
        =
        \abs{\cc}^{-1}
        V_{r,\bar s}(\bar y,\bar w)\,
        \psi\bigl(\bar s,Z_{r,\bar s}(\bar y,\bar w)\bigr).
    \]
    The support of $\psi$ forces $\bar s\in\mathcal I$ and
    $Z_{r,\bar s}(\bar y,\bar w)$ to remain in a fixed compact set.
    Since $\mathscr A_{\bar s}(\vartheta)$ is uniformly bounded on
    $\mathcal I\times\R$, this implies that $(\bar y,\bar w)$ remains in
    a fixed compact set. The uniform $\C^N$-bounds for
    $\widetilde H_r$ and $\widetilde \Theta_r$ follow from the chain rule and
    the bounds above.

    The same argument gives the uniform $\C_c^N$-bounds for
    \[
        \Phi_{r,\lambda,\bar y}(\eta)
        =
        \psi\bigl(\lambda,Z_{r,\lambda}(\bar y,\eta)\bigr).
    \]

    Finally, define the rescaled vector kernel
    \[
        \widetilde{\vec K}_r(\bar s, \bar y, \bar w)
        :=
        r^{1+\homd-(\beta-\alpha)}
        \vec K_r(\delta_r(\bar s, \bar y, \bar w)).
    \]
    By the explicit formula for $\vec K_r$, the same computation as for
    $\widetilde \Theta_r$ gives uniform $\C_c^N$-bounds for
    $\widetilde{\vec K}_r$, for every $N$. Since
    $L_r=-\nabla_w\cdot\vec K_r$, we have
    \[
        \Psi_{r,\lambda,\bar y}(\eta)
        =
        -\nabla_\eta\cdot
        \widetilde{\vec K}_r(\lambda,\bar y,\eta).
    \]
    Therefore $\Psi_{r,\lambda,\bar y}$ is uniformly bounded in
    $\C_c^N(\R^d)$, with support in one fixed ball. Its zero-average
    property follows by integrating the divergence of the compactly
    supported vector field:
    \[
        \int_{\R^d}\Psi_{r,\lambda,\bar y}(\eta)\dd\eta
        =
        -\int_{\R^d}
        \nabla_\eta\cdot
        \widetilde{\vec K}_r(\lambda,\bar y,\eta)
        \dd\eta
        =
        0.
    \]
\end{proof}

\subsubsection{Domination and kernel bounds I}

We now turn to the equation \eqref{eq:kinetic-frac-p}, namely to the
definition of the nonlocal diffusion through differences.

For an $h$-dependent function $F=F(t,x,v,h)$ and $1<s<\infty$ set
\[
    \mathcal N_sF(t,x,v)
    :=\left(\int_{\R^d}\abs{F(t,x,v,h)}^s\dd\eta(h)\right)^{1/s}
\]
and
\[
    \mathcal N_s^\#F(t,x,v)
    :=\left(\int_{\R^d}\abs{F(t,x,v-h,h)}^s\dd\eta(h)\right)^{1/s}.
\]
The change of variables $v\mapsto v-h$ gives
\begin{equation}
    \label{eq:Ns-sharp-norm}
    \|\mathcal N_sF\|_{\L^s(\R^{1+2d})}
    =
    \|\mathcal N_s^\#F\|_{\L^s(\R^{1+2d})}
    =
    \|F\|_{\L^s(\dd\eta\otimes\dd(t,x,v))}.
\end{equation}

\begin{lemma}
    \label{lem:single-scale-gagliardo-kernels}
    Let $0<\sigma<1$ and $0<\alpha<\beta<\alpha(1+1/\sigma)$.
    For every $1<s<\infty$ there exist nonnegative kernels $P_r,Q_r$ such that
    \begin{equation}
        \label{eq:PQ-L1-Linfty}
        \|P_r\|_{\L^1}+\|Q_r\|_{\L^1}\lesssim 1,
        \qquad
        \|P_r\|_{\L^\infty}+\|Q_r\|_{\L^\infty}\lesssim  r^{-\homd},
    \end{equation}
    and for every sufficiently smooth function $F = F(t,x,v,h) \colon \R^{1+2d} \times \R^d \to \R$,
    \begin{equation*}
        \abs{[T_{\frG_{r,\sigma}}^\eta F](t,x,v)}
        \lesssim
        r^{\alpha-\sigma(\beta-\alpha)-1}
        \Bigl([T_{P_r}(\mathcal N_sF)](t,x,v)+[T_{P_r}(\mathcal N_s^\#F)](t,x,v)\Bigr),
    \end{equation*}
    and
    \begin{equation}
        \label{eq:single-scale-v-kernel}
        \abs{[T_{\frG^v_{r,\sigma}}^\eta F](t,x,v)}
        \lesssim
        r^{\sigma(\beta-\alpha)-1}
        \Bigl([T_{Q_r}(\mathcal N_sF)](t,x,v)+[T_{Q_r}(\mathcal N_s^\#F)](t,x,v)\Bigr)
    \end{equation}
    for all $(t,x,v) \in \R^{1+2d}$.
\end{lemma}

\begin{proof}
    Put $R:=r^{\beta-\alpha}$.  By Lemma~\ref{lem:rescaled-profile-bounds},
    \[
        H_r(s,y,w)=r^{\alpha-1-\homd}\widetilde H_r(\delta_r^{-1}(s,y,w)),
    \]
    where $\widetilde H_r$ is uniformly bounded in $\C_c^1$.  Therefore, with
    $W:=R^{-1}w$ and $H:=R^{-1}h$,
    \[
        \abs{\frG_{r,\sigma}(s,y,w,h)}
        \le
        C r^{\alpha-1-\homd}R^{-\sigma}\,
        \Omega_r\Bigl(\frac{s}{r^\alpha},\frac{y}{r^\beta},W,H\Bigr),
    \]
    where $\Omega_r$ is supported in a fixed compact set of the
    $(\bar s,\bar y)$ variables and satisfies
    \[
        \begin{aligned}
            \Omega_r(\bar s,\bar y,W,H)
             & \le C\Bigl(
            \abs{H}^{1-\sigma}\one_{\{\abs{H}\le1,\ \abs{W}\le C\}}
            +\abs{H}^{-\sigma}\one_{\{\abs{H}>1,\ \abs{W}\le C\}}
            \\
             & \hphantom{\le C\Bigl(}
            +\abs{H}^{-\sigma}\one_{\{\abs{H}>1,\ \abs{W+H}\le C\}}
            \Bigr).
        \end{aligned}
    \]
    The first term follows from the mean-value theorem, and the last two from
    the trivial difference bound.

    Let $s'=s/(s-1)$. Let $E\subset \R^{1+d}$ be a fixed compact set containing the support of
    $\Omega_r$ in the $(\bar s,\bar y)$-variables, and choose
    $C_0\ge 1$ so that the above bound for $\Omega_r$ holds with
    $C_0$ in place of the implicit constants.  Thus, after increasing the
    constant $C$ if necessary,
    \[
        \Omega_r(\bar s,\bar y,W,H)
        \le
        C\one_E(\bar s,\bar y)
        \bigl(\Omega^{(1)}(W,H)+\Omega^{(2)}(W,H)+\Omega^{(3)}(W,H)\bigr),
    \]
    where
    \[
        \Omega^{(1)}(W,H)
        :=
        \abs{H}^{1-\sigma}\one_{\{\abs{H}\le 1,\ \abs{W}\le C_0\}},
    \]
    \[
        \Omega^{(2)}(W,H)
        :=
        \abs{H}^{-\sigma}\one_{\{\abs{H}> 1,\ \abs{W}\le C_0\}},
    \]
    and
    \[
        \Omega^{(3)}(W,H)
        :=
        \abs{H}^{-\sigma}\one_{\{\abs{H}> 1,\ \abs{W+H}\le C_0\}}.
    \]
    Since $\dd\eta(h)=\dd\eta(H)$ under the change of variables
    $h=RH$, Hölder's inequality in $H$ gives, for fixed
    $(s,y,w)$,
    \begin{align*}
         & \int_{\R^d}
        \abs{F((t,x,v)\circ (s,y,w),h)}
        \Omega^{(1)}(W,H)\dd\eta(h) \\
         & \le
        \mathcal N_sF((t,x,v)\circ (s,y,w))
        \one_{\{\abs{W}\le C_0\}}
        \left(
        \int_{\abs{H}\le 1}
        \abs{H}^{(1-\sigma)s'}\dd\eta(H)
        \right)^{1/s'} .
    \end{align*}
    The last factor is finite, because
    $\dd\eta(H)=\abs{H}^{-d}\dd H$ and $0<\sigma<1$.  Similarly,
    \begin{align*}
         & \int_{\R^d}
        \abs{F((t,x,v)\circ (s,y,w),h)}
        \Omega^{(2)}(W,H)\dd\eta(h) \\
         & \le
        \mathcal N_sF((t,x,v)\circ (s,y,w))
        \one_{\{\abs{W}\le C_0\}}
        \left(
        \int_{\abs{H}>1}
        \abs{H}^{-\sigma s'}\dd\eta(H)
        \right)^{1/s'} ,
    \end{align*}
    and this last factor is finite as well.  Hence the first two pieces of
    $\Omega_r$ are bounded by
    \[
        C\,
        \one_E\left(\frac{s}{r^\alpha},\frac{y}{r^\beta}\right)
        \one_{\{\abs{w}\le C_0R\}}
        \mathcal N_sF((t,x,v)\circ(s,y,w)).
    \]

    It remains to treat the third piece.  In the contribution of
    $\Omega^{(3)}$, make the change of variables
    \[
        w'=w+h,
        \qquad
        W'=W+H.
    \]
    Since $h=RH$, this change has unit Jacobian in the $w$-variable, and
    the support condition $\abs{W+H}\le C_0$ becomes
    $\abs{W'}\le C_0$.  Moreover,
    \[
        (t,x,v)\circ(s,y,w)
        =
        \bigl(t+s,x+y+sv,v+w'-h\bigr),
    \]
    while
    \[
        (t,x,v)\circ(s,y,w')
        =
        \bigl(t+s,x+y+sv,v+w'\bigr).
    \]
    Thus the function $F$ is replaced by the shifted expression appearing
    in $\mathcal N_s^\#F$.  Applying Hölder in $H$ gives
    \begin{align*}
         & \int_{\R^d}
        \abs{F(t+s,x+y+sv,v+w'-h,h)}
        \abs{H}^{-\sigma}\one_{\{\abs{H}>1\}}\dd\eta(h) \\
         & \le
        C\,
        \mathcal N_s^\#F((t,x,v)\circ(s,y,w')),
    \end{align*}
    because
    \[
        \left(
        \int_{\abs{H}>1}
        \abs{H}^{-\sigma s'}\dd\eta(H)
        \right)^{1/s'}
        <\infty.
    \]
    Therefore the third piece produces the same compactly supported kernel,
    now acting on $\mathcal N_s^\#F$.

    Define
    \[
        P_r^0(\bar s,\bar y,W)
        :=
        \one_E(\bar s,\bar y)\one_{\{\abs{W}\le C_0\}},
    \]
    and
    \[
        P_r(s,y,w)
        :=
        r^{-\homd}
        P_r^0\left(\frac{s}{r^\alpha},
        \frac{y}{r^\beta},
        \frac{w}{r^{\beta-\alpha}}\right)
        =
        r^{-\homd}P_r^0(\delta_r^{-1}(s,y,w)).
    \]
    Combining the three contributions yields
    \[
        \abs{[T_{\frG_{r,\sigma}}^\eta F](t,x,v)}
        \le
        C r^{\alpha-\sigma(\beta-\alpha)-1}
        \Bigl(
        [T_{P_r}(\mathcal N_sF)](t,x,v)
        +
        [T_{P_r}(\mathcal N_s^\#F)](t,x,v)
        \Bigr).
    \]
    The family $P_r^0$ is uniformly bounded and supported in the fixed
    compact set $E\times B_{C_0}$.  Hence
    \[
        \|P_r\|_{\L^1}\le C,
        \qquad
        \|P_r\|_{\L^\infty}\le C r^{-\homd},
    \]
    which is the $P_r$-part of \eqref{eq:PQ-L1-Linfty}.

    We next treat $\frG^v_{r,\sigma}$.  Let
    \[
        U_r(s,y,w):=\frac1{c_{d,\sigma}^{\mathrm{Gag}}}
        (-\Delta_w)^{-\sigma}(\nabla_w\!\cdot\Theta_r)(s,y,w),
    \]
    so that $\frG^v_{r,\sigma}=\Dfr_w U_r$.  Again by
    Lemma~\ref{lem:rescaled-profile-bounds},
    \[
        \Theta_r(s,y,w)=r^{\beta-\alpha-1-\homd}
        \widetilde\Theta_r(\delta_r^{-1}(s,y,w)),
    \]
    with $\widetilde\Theta_r$ uniformly bounded in $\C_c^N$ for every fixed
    $N$.  Hence
    \[
        U_r(\delta_r(\bar s, \bar y, \bar w))
        =r^{-1-\homd+2\sigma(\beta-\alpha)}\widetilde U_r(\bar s, \bar y, \bar w),
        \qquad
        \widetilde U_r:=\frac1{c_{d,\sigma}^{\mathrm{Gag}}}
        (-\Delta_{\bar w})^{-\sigma}
        (\nabla_{\bar w}\!\cdot\widetilde\Theta_r).
    \]
    The kernel bounds of Lemma~\ref{lem:potential-divergence-decay} for the potential
    $(-\Delta_{\bar w})^{-\sigma}\nabla_{\bar w}\cdot$ applied to the uniformly
    compactly supported $\C^N$ family give, uniformly in $r$,
    \begin{equation}
        \label{eq:Ur-profile-bounds}
        \abs{\widetilde U_r(\bar s,\bar y,W)}
        \le
        C\one_E(\bar s,\bar y)(1+\abs{W})^{-d-1+2\sigma},
    \end{equation}
    and
    \begin{equation}
        \label{eq:grad-Ur-profile-bounds}
        \abs{\nabla_W\widetilde U_r(\bar s,\bar y,W)}
        \le
        C\one_E(\bar s,\bar y)(1+\abs{W})^{-d-2+2\sigma},
    \end{equation}
    where $E\subset\R^{1+d}$ is fixed and compact.

    Therefore
    \[
        \abs{\frG^v_{r,\sigma}(s,y,w,h)}
        \le
        C r^{-1-\homd+\sigma(\beta-\alpha)}
        \frac{\abs{\widetilde U_r(\bar s,\bar y,W+H)
                -\widetilde U_r(\bar s,\bar y,W)}}{\abs{H}^\sigma},
    \]
    where $(\bar s,\bar y,W)=(s/r^\alpha,y/r^\beta,w/R)$ and $H=h/R$.
    Set $L_W:=(1+\abs{W})/2$.  If $\abs{H}\le L_W$, then the mean-value
    theorem and \eqref{eq:grad-Ur-profile-bounds} imply
    \[
        \frac{\abs{\widetilde U_r(W+H)-\widetilde U_r(W)}}{\abs{H}^\sigma}
        \le
        C\abs{H}^{1-\sigma}(1+\abs{W})^{-d-2+2\sigma}.
    \]
    If $\abs{H}>L_W$, then \eqref{eq:Ur-profile-bounds} gives
    \[
        \frac{\abs{\widetilde U_r(W+H)-\widetilde U_r(W)}}{\abs{H}^\sigma}
        \le
        C\abs{H}^{-\sigma}
        \Bigl((1+\abs{W})^{-d-1+2\sigma}
        +(1+\abs{W+H})^{-d-1+2\sigma}\Bigr).
    \]
    Let $s'=s/(s-1)$.  Since
    \[
        \left(\int_{\abs{H}\le L_W}\abs{H}^{(1-\sigma)s'}\dd\eta(H)
        \right)^{1/s'}
        \le C(1+\abs{W})^{1-\sigma}
    \]
    and
    \[
        \left(\int_{\abs{H}> L_W}\abs{H}^{-\sigma s'}\dd\eta(H)
        \right)^{1/s'}
        \le C(1+\abs{W})^{-\sigma},
    \]
    the terms depending on $W$ give, after H\"older's inequality in $H$,
    the kernel
    \[
        r^{-\homd}\one_E\left(\frac{s}{r^\alpha},\frac{y}{r^\beta}\right)
        \left(1+\frac{\abs{w}}{R}\right)^{-d-1+\sigma}.
    \]
    For the term depending on $W+H$ we change variables $w'=w+h$.  Then
    $W'=W+H$, and the condition $\abs{H}>(1+\abs{W})/2$ implies $\abs{H}\ge c(1+\abs{W'})$.
    Hence
    \[
        \left(\int_{\abs{H}\ge c(1+\abs{W'})}\abs{H}^{-\sigma s'}\dd\eta(H)
        \right)^{1/s'}
        \le C(1+\abs{W'})^{-\sigma},
    \]
    and this gives the same kernel with $w'$ in place of $w$, acting on
    $\mathcal N_s^\#F$.

    Thus \eqref{eq:single-scale-v-kernel} holds with
    \[
        Q_r(s,y,w)
        :=r^{-\homd}
        \one_E\left(\frac{s}{r^\alpha},\frac{y}{r^\beta}\right)
        \left(1+\frac{\abs{w}}{r^{\beta-\alpha}}\right)^{-d-1+\sigma}.
    \]
    Because $d+1-\sigma>d$, this kernel has uniformly bounded $\L^1$ norm,
    and its $\L^\infty$ norm is bounded by $Cr^{-\homd}$.  This proves
    \eqref{eq:PQ-L1-Linfty} and finishes the proof.
\end{proof}

\begin{lemma}
    \label{lem:integrated-gagliardo-kernels}
    Let $0<\sigma<1$, $1<p<\infty$, and
    $0<\alpha<\beta<\alpha(1+1/\sigma)$.  Then, uniformly in $\tau>0$,
    \begin{equation}
        \label{eq:integrated-v-estimate}
        \left\|T_{\int_0^\tau \frG^v_{r,\sigma}\dd r}^\eta F\right\|_{\L^q}
        \le C\|F\|_{\L^p(\dd\eta\otimes\dd(t,x,v))}
    \end{equation}
    whenever $1<q<\infty$ and
    \[
        \frac1q=\frac1p-\frac{\sigma(\beta-\alpha)}{\homd}.
    \]
    Moreover,
    \begin{equation}
        \label{eq:integrated-source-estimate}
        \left\|T_{\int_0^\tau \frG_{r,\sigma}\dd r}^\eta F\right\|_{\L^q}
        \le C\|F\|_{\L^{p'}(\dd\eta\otimes\dd(t,x,v))}
    \end{equation}
    whenever $1<q<\infty$ and
    \[
        \frac1q=\frac1{p'}-\frac{\alpha-\sigma(\beta-\alpha)}{\homd}.
    \]
\end{lemma}

\begin{proof}
    We first prove the velocity estimate.  Since $\beta>\alpha$, we have
    $\sigma(\beta-\alpha)>0$. Moreover, for every admissible $q$ in the
    statement,
    \[
        \frac1q=\frac1p-\frac{\sigma(\beta-\alpha)}{\homd}>0,
    \]
    and hence $\sigma(\beta-\alpha)<\homd$.

    By Lemma~\ref{lem:single-scale-gagliardo-kernels},
    \[
        \abs{T_{\int_0^\tau \frG^v_{r,\sigma}\dd r}^\eta F}
        \le
        C\Bigl(
        T_{\int_0^\tau r^{\sigma(\beta-\alpha)-1}Q_r\dd r}(\mathcal N_pF)
        +
        T_{\int_0^\tau r^{\sigma(\beta-\alpha)-1}Q_r\dd r}(\mathcal N_p^\#F)
        \Bigr).
    \]
    We now estimate the integrated kernel using
    Lemma~\ref{lem:critical-integration-interpolation}.  Put $J_r:=r^{\sigma(\beta-\alpha)-1}Q_r$.
    Since the kernels $Q_r$ satisfy
    \[
        \|Q_r\|_{\L^1}\lesssim 1,
        \qquad
        \|Q_r\|_{\L^\infty}\lesssim r^{-\homd},
    \]
    we obtain
    \[
        \|J_r\|_{\L^1}
        \lesssim r^{\sigma(\beta-\alpha)-1},
    \]
    and
    \[
        \|J_r\|_{\L^\infty}
        \lesssim r^{\sigma(\beta-\alpha)-1-\homd}
        =
        r^{-1-(\homd-\sigma(\beta-\alpha))}.
    \]
    Hence Lemma~\ref{lem:critical-integration-interpolation}, applied with
    \[
        N=1+2d,
        \qquad
        \mu=\sigma(\beta-\alpha),
        \qquad
        \nu=\homd-\sigma(\beta-\alpha),
    \]
    gives, uniformly in $\tau>0$,
    \[
        \left\|
        \int_0^\tau r^{\sigma(\beta-\alpha)-1}Q_r\dd r
        \right\|_{\L^{\theta_v,\infty}}
        \lesssim 1,
        \qquad
        \theta_v
        =
        \frac{\sigma(\beta-\alpha)+\homd-\sigma(\beta-\alpha)}
        {\homd-\sigma(\beta-\alpha)}
        =
        \frac{\homd}{\homd-\sigma(\beta-\alpha)}.
    \]
    Thus
    \[
        \frac1{\theta_v}
        =
        1-\frac{\sigma(\beta-\alpha)}{\homd}.
    \]
    Consequently,
    \[
        \frac1q+1
        =
        \frac1p+1-\frac{\sigma(\beta-\alpha)}{\homd}
        =
        \frac1p+\frac1{\theta_v},
    \]
    which is the exponent relation needed for the weak Young
    inequality in Lemma~\ref{lem:young-nl}. Therefore,
    \[
        \left\|T_{\int_0^\tau \frG^v_{r,\sigma}\dd r}^\eta F\right\|_{\L^q}
        \lesssim
        \|\mathcal N_pF\|_{\L^p}
        +
        \|\mathcal N_p^\#F\|_{\L^p}.
    \]
    Using \eqref{eq:Ns-sharp-norm}, we obtain
    \[
        \left\|T_{\int_0^\tau \frG^v_{r,\sigma}\dd r}^\eta F\right\|_{\L^q}
        \lesssim
        \|F\|_{\L^p(\dd\eta\otimes\dd(t,x,v))},
    \]
    which proves \eqref{eq:integrated-v-estimate}.

    The source estimate is analogous.  The assumption
    \[
        \beta<\alpha\left(1+\frac1\sigma\right)
    \]
    is equivalent to $\sigma(\beta-\alpha)<\alpha$,
    and therefore $\alpha-\sigma(\beta-\alpha)>0$. Moreover, for every
    admissible $q$,
    \[
        \frac1q=\frac1{p'}-\frac{\alpha-\sigma(\beta-\alpha)}{\homd}>0,
    \]
    hence $\alpha-\sigma(\beta-\alpha)<\homd$.

    By Lemma~\ref{lem:single-scale-gagliardo-kernels}, now with
    $s=p'$,
    \[
        \abs{T_{\int_0^\tau \frG_{r,\sigma}\dd r}^\eta F}
        \le
        C\Bigl(
        T_{\int_0^\tau r^{\alpha-\sigma(\beta-\alpha)-1}P_r\dd r}(\mathcal N_{p'}F)
        +
        T_{\int_0^\tau r^{\alpha-\sigma(\beta-\alpha)-1}P_r\dd r}(\mathcal N_{p'}^\#F)
        \Bigr).
    \]
    Put $J_r:=r^{\alpha-\sigma(\beta-\alpha)-1}P_r$.
    As before,
    \[
        \|J_r\|_{\L^1}
        \lesssim r^{\alpha-\sigma(\beta-\alpha)-1},
    \]
    and
    \[
        \|J_r\|_{\L^\infty}
        \lesssim r^{\alpha-\sigma(\beta-\alpha)-1-\homd}
        =
        r^{-1-(\homd-\alpha+\sigma(\beta-\alpha))}.
    \]
    Applying Lemma~\ref{lem:critical-integration-interpolation} with
    \[
        N=1+2d,
        \qquad
        \mu=\alpha-\sigma(\beta-\alpha),
        \qquad
        \nu=\homd-\alpha+\sigma(\beta-\alpha),
    \]
    yields
    \[
        \left\|
        \int_0^\tau r^{\alpha-\sigma(\beta-\alpha)-1}P_r\dd r
        \right\|_{\L^{\theta_S,\infty}}
        \lesssim 1,
        \qquad
        \theta_S
        =
        \frac{\homd}{\homd-\alpha+\sigma(\beta-\alpha)}.
    \]
    Hence
    \[
        \frac1{\theta_S}
        =
        1-\frac{\alpha-\sigma(\beta-\alpha)}{\homd},
    \]
    and therefore
    \[
        \frac1q+1
        =
        \frac1{p'}+1-\frac{\alpha-\sigma(\beta-\alpha)}{\homd}
        =
        \frac1{p'}+\frac1{\theta_S}.
    \]
    Lemma~\ref{lem:young-nl} gives
    \[
        \left\|T_{\int_0^\tau \frG_{r,\sigma}\dd r}^\eta F\right\|_{\L^q}
        \lesssim
        \|\mathcal N_{p'}F\|_{\L^{p'}}
        +
        \|\mathcal N_{p'}^\#F\|_{\L^{p'}}.
    \]
    Finally, by \eqref{eq:Ns-sharp-norm},
    \[
        \left\|T_{\int_0^\tau \frG_{r,\sigma}\dd r}^\eta F\right\|_{\L^q}
        \lesssim
        \|F\|_{\L^{p'}(\dd\eta\otimes\dd(t,x,v))},
    \]
    which proves \eqref{eq:integrated-source-estimate}.
\end{proof}

\subsubsection{Kernel bounds II}

Let us now derive the kernel bounds in the Bessel case. The next proposition is the key place where the
nonlocality plays a role.

\begin{prop}
    \label{prop:kernel-pointwise}
    There exists $C>0$ such that for every $r>0$,
    \begin{equation}
        \label{eq:pointwise-bessel-G}
        \abs{G_{r,\sigma}(s,y,w)}
        \le C r^{\alpha-1-\homd-\sigma(\beta-\alpha)}
        \one_{\{\abs{s}\sim r^\alpha,\ \abs{y}\lesssim r^\beta\}}
        \Bigl(1+\frac{\abs{w}}{r^{\beta-\alpha}}\Bigr)^{-d-\sigma},
    \end{equation}
    and
    \begin{equation}
        \label{eq:pointwise-bessel-Gv}
        \abs{G^v_{r,\sigma}(s,y,w)}
        \le C r^{-1-\homd+\sigma(\beta-\alpha)}
        \one_{\{\abs{s}\sim r^\alpha,\ \abs{y}\lesssim r^\beta\}}
        \Bigl(1+\frac{\abs{w}}{r^{\beta-\alpha}}\Bigr)^{-d-1+\sigma}.
    \end{equation}
\end{prop}

\begin{proof}
    We first treat $G_{r,\sigma}$.
    Fix $r>0$ and $(s,y)$ with $\abs{s}\le 2r^\alpha$ and $\abs{y}\le Cr^\beta$.
    Define
    \[
        \Phi_{r,s,y}(\eta)
        :=\psi\Bigl(\frac{s}{r^\alpha},\A_{s/r^\alpha}(r)^{-1}\binom{y}{r^{\beta-\alpha}\eta}\Bigr),
        \qquad \eta\in\R^d.
    \]
    Put
    \[
        \lambda:=\frac{s}{r^\alpha},
        \qquad
        \bar y:=r^{-\beta}y.
    \]

    If $\lambda$ does not belong to the time projection of $\supp\psi$, or
    if $\bar y$ is outside the corresponding fixed support set, then the
    kernel vanishes and there is nothing to prove. Otherwise
    \[
        \Phi_{r,s,y}(\eta)
        =
        \Phi_{r,\lambda,\bar y}(\eta)
    \]
    in the notation of Lemma~\ref{lem:rescaled-profile-bounds}. Hence the
    family $\{\Phi_{r,s,y}\}$ is bounded in $\C_c^1(\R^d)$, uniformly in
    $r,s,y$, and all supports are contained in one fixed ball.
    By the scaling of the fractional derivative,
    \[
        D_w^\sigma\Bigl[\psi\Bigl(\frac{s}{r^\alpha},\A_{s/r^\alpha}(r)^{-1}\binom{y}{w}\Bigr)\Bigr]
        =
        r^{-\sigma(\beta-\alpha)} (D^\sigma \Phi_{r,s,y})(r^{\alpha-\beta}w).
    \]
    Lemma~\ref{lem:family-fractional}(i) therefore yields
    \[
        \abs{D_w^\sigma\Bigl[\psi\Bigl(\frac{s}{r^\alpha},\A_{s/r^\alpha}(r)^{-1}\binom{y}{w}\Bigr)\Bigr]}
        \lesssim r^{-\sigma(\beta-\alpha)}\Bigl(1+\frac{\abs{w}}{r^{\beta-\alpha}}\Bigr)^{-d-\sigma}.
    \]
    Multiplying by the prefactor in \eqref{eq:def-bessel-G-kernel} and using $\abs{s}\sim r^\alpha$ proves
    \eqref{eq:pointwise-bessel-G}.

    We now turn to $G^v_{r,\sigma}$.
    Put again
    \[
        \lambda:=\frac{s}{r^\alpha},
        \qquad
        \bar y:=r^{-\beta}y.
    \]
    If the kernel does not vanish, then $\lambda$ and $\bar y$ lie in the
    fixed support set described in Lemma~\ref{lem:rescaled-profile-bounds}.
    Moreover,
    \[
        \Psi_{r,s,y}(\eta)
        :=
        r^{1+\homd}L_r(s,y,r^{\beta-\alpha}\eta)
        =
        \Psi_{r,\lambda,\bar y}(\eta).
    \]
    Therefore Lemma~\ref{lem:rescaled-profile-bounds} implies that
    $\Psi_{r,s,y}$ is uniformly bounded in $\C_c^1(\R^d)$, supported in a
    fixed ball, and has zero average.
    By the scaling of the Riesz potential,
    \[
        G^v_{r,\sigma}(s,y,w)
        =I_w^\sigma L_r(s,y,\cdot)(w)
        =r^{-1-\homd+\sigma(\beta-\alpha)}
        \bigl(I^\sigma\Psi_{r,s,y}\bigr)(r^{\alpha-\beta}w).
    \]
    Lemma~\ref{lem:family-fractional}~(ii) gives
    \[
        \abs{G^v_{r,\sigma}(s,y,w)}
        \lesssim r^{-1-\homd+\sigma(\beta-\alpha)}
        \Bigl(1+\frac{\abs{w}}{r^{\beta-\alpha}}\Bigr)^{-d-1+\sigma}.
    \]
    Since $L_r$ vanishes unless $\abs{s}\sim r^\alpha$ and $\abs{y}\lesssim r^\beta$, this proves
    \eqref{eq:pointwise-bessel-Gv}.
\end{proof}

The pointwise bounds immediately imply the corresponding $\L^\theta$ estimates.

\begin{lemma}
    \label{lem:kernel-Ltheta}
    Let $\theta\in[1,\infty]$.
    Then, uniformly in $r>0$,
    \[
        \|G^v_{r,\sigma}\|_{\L^\theta}
        \lesssim r^{\sigma(\beta-\alpha)-1+\homd(\frac1\theta-1)},
        \qquad
        \|G_{r,\sigma}\|_{\L^\theta}
        \lesssim r^{\alpha-\sigma(\beta-\alpha)-1+\homd(\frac1\theta-1)}.
    \]
\end{lemma}

\begin{proof}
    For $G^v_{r,\sigma}$ and $G_{r,\sigma}$ we integrate the bounds from
    Proposition~\ref{prop:kernel-pointwise}.
    Indeed, the polynomial tails in the $w$-variable are integrable in
    $\L^\theta_w$ for every $1\le \theta<\infty$, since
    \[
        \theta(d+\sigma)>d
        \qquad\text{and}\qquad
        \theta(d+1-\sigma)>d
    \]
    whenever $0<\sigma<1$.
    For instance,
    \begin{align*}
        \|G^v_{r,\sigma}\|_{\L^\theta}^\theta
         & \lesssim r^{\theta(-1-\homd+\sigma(\beta-\alpha))}
        \cdot r^{\alpha+\beta d}
        \cdot \int_{\R^d}\Bigl(1+\frac{\abs{w}}{r^{\beta-\alpha}}\Bigr)^{-\theta(d+1-\sigma)}\dd w \\
         & \lesssim r^{\theta(-1-\homd+\sigma(\beta-\alpha))}
        \cdot r^{\alpha+\beta d}
        \cdot r^{(\beta-\alpha)d}                                                                  \\
         & = r^{\theta(-1-\homd+\sigma(\beta-\alpha))+\homd},
    \end{align*}
    which yields the desired exponent after taking the $\theta$-th root. The case $\theta=\infty$ follows
    directly from Proposition~\ref{prop:kernel-pointwise}.
    The estimate for $G_{r,\sigma}$ is identical.
\end{proof}

\begin{lemma}
    \label{lem:int-kernels}
    Assume $0<\sigma<1$ and $0<\alpha<\beta<\alpha(1+1/\sigma)$.
    Define
    \[
        \theta_1:=\frac{\homd}{\homd-\sigma(\beta-\alpha)},
        \qquad
        \theta_2:=\frac{\homd}{\homd-\alpha+\sigma(\beta-\alpha)}.
    \]
    Then, uniformly in $\tau>0$,
    \[
        \norm{\int_0^\tau G^v_{r,\sigma}\dd r}_{\L^{\theta_1,\infty}(\R^{1+2d})}\lesssim1,
        \qquad
        \norm{\int_0^\tau G_{r,\sigma}\dd r}_{\L^{\theta_2,\infty}(\R^{1+2d})}\lesssim1.
    \]
\end{lemma}

\begin{proof}
    By Lemma~\ref{lem:kernel-Ltheta}, evaluated at $\theta=1$ and
    $\theta=\infty$,
    \[
        \|G^v_{r,\sigma}\|_{\L^1}
        \lesssim
        r^{-1+\sigma(\beta-\alpha)},
        \qquad
        \|G^v_{r,\sigma}\|_{\L^\infty}
        \lesssim
        r^{-1-\bigl(\homd-\sigma(\beta-\alpha)\bigr)}.
    \]
    Apply Lemma~\ref{lem:critical-integration-interpolation} with $\mu=\sigma(\beta-\alpha)$ and $\nu=\homd-\sigma(\beta-\alpha)$.
    Since
    \[
        \frac{\mu+\nu}{\nu}
        =
        \frac{\homd}{\homd-\sigma(\beta-\alpha)}
        =
        \theta_1,
    \]
    we obtain
    \[
        \left\|\int_0^\tau G^v_{r,\sigma}\dd r\right\|_{\L^{\theta_1,\infty}}
        \lesssim 1
        \qquad\text{uniformly in }\tau>0.
    \]

    Likewise, Lemma~\ref{lem:kernel-Ltheta} gives
    \[
        \|G_{r,\sigma}\|_{\L^1}
        \lesssim
        r^{-1+\alpha-\sigma(\beta-\alpha)},
        \qquad
        \|G_{r,\sigma}\|_{\L^\infty}
        \lesssim
        r^{-1-\bigl(\homd-\alpha+\sigma(\beta-\alpha)\bigr)}.
    \]
    Here $\alpha-\sigma(\beta-\alpha)>0$ by the assumption $\beta<\alpha(1+1/\sigma)$.
    Applying Lemma~\ref{lem:critical-integration-interpolation} with
    \[
        \mu=\alpha-\sigma(\beta-\alpha),
        \qquad
        \nu=\homd-\alpha+\sigma(\beta-\alpha),
    \]
    yields
    \[
        \left\|\int_0^\tau G_{r,\sigma}\dd r\right\|_{\L^{\theta_2,\infty}}
        \lesssim 1,
        \qquad
        \theta_2
        =
        \frac{\homd}{\homd-\alpha+\sigma(\beta-\alpha)}.
    \]
    This proves the lemma.
\end{proof}

\section{Proofs of the kinetic nonlocal Gagliardo--Nirenberg inequalities}
\label{sec:gain}

\subsection{Nonlocal diffusion of Gagliardo type}

\begin{proof}[Proof of Theorem~\ref{thm:int:mainGN-differencemp}]
    Choose $\alpha>0$ and put $\beta=\rho\alpha$, where
    \begin{equation}
        \label{eq:rho-gagliardo-proof}
        \rho:=
        \frac{d(p-2)+2p\sigma+2}{2(d(p-2)+p\sigma)}.
    \end{equation}
    Since
    \[
        d(p-2)+p\sigma=p(d+\sigma)-2d>0
    \]
    in the range of the theorem, this is well defined.  Moreover,
    \[
        \rho-1
        =
        \frac{2d+2-dp}{2(d(p-2)+p\sigma)},
    \]
    so $\rho>1$ is equivalent to $p<2+2/d$.  Also
    \[
        1-\sigma(\rho-1)
        =
        \frac{p(2d+d\sigma+2\sigma)-(4d+2d\sigma+2\sigma)}
        {2(d(p-2)+p\sigma)},
    \]
    and therefore $1-\sigma(\rho-1)>0$ is equivalent to
    \[
        p>2-\frac{2\sigma}{2d+d\sigma+2\sigma}.
    \]
    Thus $0<\alpha<\beta<\alpha(1+1/\sigma)$.

    For this choice, we have
    \[
        \homd=\alpha\bigl((2\rho-1)d+1\bigr),
    \]
    and Proposition~\ref{prop:representation-formula-nonlocal} yields
    \[
        f-T_{K_\tau}f
        =
        T_{\int_0^\tau \frG_{r,\sigma}\dd r}^\eta S
        +
        T_{\int_0^\tau \frG^v_{r,\sigma}\dd r}^\eta(\Dv f).
    \]
    Lemma~\ref{lem:integrated-gagliardo-kernels} gives the two estimates
    \[
        \left\|T_{\int_0^\tau \frG^v_{r,\sigma}\dd r}^\eta(\Dv f)\right\|_{\L^q}
        \le C\|\Dv f\|_{\L^p(\dd\eta\otimes\dd(t,x,v))}
    \]
    provided
    \begin{equation}
        \label{eq:q-v-condition-gagliardo}
        \frac1q+\frac{\sigma(\rho-1)}{(2\rho-1)d+1}=\frac1p,
    \end{equation}
    and
    \[
        \left\|T_{\int_0^\tau \frG_{r,\sigma}\dd r}^\eta S\right\|_{\L^q}
        \le C\|S\|_{\L^{p'}(\dd\eta\otimes\dd(t,x,v))}
    \]
    provided
    \begin{equation}
        \label{eq:q-source-condition-gagliardo}
        \frac1q+
        \frac{1-\sigma(\rho-1)}{(2\rho-1)d+1}
        =\frac1{p'}.
    \end{equation}
    The value of $\rho$ in \eqref{eq:rho-gagliardo-proof} is exactly the
    solution for which \eqref{eq:q-v-condition-gagliardo} and
    \eqref{eq:q-source-condition-gagliardo} give the same $q$.  Substituting
    \eqref{eq:rho-gagliardo-proof} into \eqref{eq:q-v-condition-gagliardo}
    gives
    \begin{equation}
        \label{eq:q-gagliardo-proof}
        \frac1q=\frac{d(p\sigma+2)}{2p(d\sigma+d+\sigma)},
        \qquad\text{that is,}\qquad
        q=\frac{2p(d\sigma+d+\sigma)}{d(p\sigma+2)}.
    \end{equation}
    Furthermore,
    \[
        q>p
        \quad\Longleftrightarrow\quad
        p<2+\frac2d,
    \]
    and
    \[
        q>p'
        \quad\Longleftrightarrow\quad
        p>2-\frac{2\sigma}{2d+d\sigma+2\sigma}.
    \]
    Hence all weak Young exponents used above are admissible, and for every
    $\tau>0$,
    \begin{equation}
        \label{eq:additive-gagliardo-before-limit}
        \|f-T_{K_\tau}f\|_{\L^q}
        \le
        C\Bigl(
        \|\Dv f\|_{\L^p(\dd\eta\otimes\dd(t,x,v))}
        +
        \|S\|_{\L^{p'}(\dd\eta\otimes\dd(t,x,v))}
        \Bigr).
    \end{equation}

    Let $\theta$ be defined by
    \[
        \frac1q+1=\frac1\theta+\frac1p.
    \]
    By Lemmas~\ref{lem:young-nl} and~\ref{lem:kernel-Kr},
    \[
        \|T_{K_\tau}f\|_{\L^q}
        \le
        \|K_\tau\|_{\L^\theta}\|f\|_{\L^p}
        \le
        C\tau^{\homd(1/q-1/p)}\|f\|_{\L^p}.
    \]
    Since $q>p$, the exponent $\homd(1/q-1/p)$ is negative.  Therefore
    \[
        \|T_{K_\tau}f\|_{\L^q}\to0
        \qquad\text{as }\tau\to\infty.
    \]
    Passing to the limit in \eqref{eq:additive-gagliardo-before-limit} gives
    the additive estimate
    \begin{equation}
        \label{eq:additive-gagliardo-estimate}
        \|f\|_{\L^q}
        \le
        C\Bigl(
        \|\Dv f\|_{\L^p(\dd\eta\otimes\dd(t,x,v))}
        +
        \|S\|_{\L^{p'}(\dd\eta\otimes\dd(t,x,v))}
        \Bigr).
    \end{equation}

    It remains to balance the two terms in
    \eqref{eq:additive-gagliardo-estimate}.  If $S=0$, then
    $f=0$ and there is nothing to prove.  We may therefore assume that
    $S\neq 0$ and $f \neq 0$.

    For $\nu>0$, define
    \[
        f_\nu(t,x,v):=f(\nu t,\nu x,v),
        \qquad
        S_\nu(t,x,v,h):=\nu S(\nu t,\nu x,v,h).
    \]
    Then
    \[
        (\partial_t+v\cdot\nabla_x)f_\nu
        =
        \nu[(\partial_t+v\cdot\nabla_x)f](\nu t,\nu x,v)
        =
        \Dvs S_\nu .
    \]
    Moreover,
    \[
        \|f_\nu\|_{\L^q}
        =
        \nu^{-\frac{d+1}{q}}\|f\|_{\L^q},
    \]
    \[
        \|\Dv f_\nu\|_{\L^p(\dd\eta\otimes\dd(t,x,v))}
        =
        \nu^{-\frac{d+1}{p}}
        \|\Dv f\|_{\L^p(\dd\eta\otimes\dd(t,x,v))},
    \]
    and
    \[
        \|S_\nu\|_{\L^{p'}(\dd\eta\otimes\dd(t,x,v))}
        =
        \nu^{1-\frac{d+1}{p'}}
        \|S\|_{\L^{p'}(\dd\eta\otimes\dd(t,x,v))}.
    \]
    Applying \eqref{eq:additive-gagliardo-estimate} to
    $(f_\nu,S_\nu)$ and multiplying by $\nu^{(d+1)/q}$ gives
    \begin{equation}
        \label{eq:nu-balance-gagliardo}
        \|f\|_{\L^q}
        \le
        C\Bigl(
        \nu^A
        \|\Dv f\|_{\L^p(\dd\eta\otimes\dd(t,x,v))}
        +
        \nu^B
        \|S\|_{\L^{p'}(\dd\eta\otimes\dd(t,x,v))}
        \Bigr),
        \qquad \nu>0,
    \end{equation}
    where
    \[
        A:=(d+1)\Bigl(\frac1q-\frac1p\Bigr),
        \qquad
        B:=1+(d+1)\Bigl(\frac1q-\frac1{p'}\Bigr).
    \]
    Using \eqref{eq:q-gagliardo-proof}, we obtain
    \[
        A
        =
        -\frac{\sigma(d+1)(2d+2-dp)}
        {2p(d\sigma+d+\sigma)}
        <0,
    \]
    and
    \[
        B
        =
        \frac{(2d+2-dp)(d\sigma+2d+\sigma)}
        {2p(d\sigma+d+\sigma)}
        >0,
    \]
    since $p<2+2/d$.  Hence the two powers of $\nu$ in
    \eqref{eq:nu-balance-gagliardo} have opposite signs.

    Optimising \eqref{eq:nu-balance-gagliardo} in $\nu>0$ yields
    \[
        \|f\|_{\L^q}
        \le
        C
        \|\Dv f\|_{\L^p(\dd\eta\otimes\dd(t,x,v))}^{\frac{B}{B-A}}
        \|S\|_{\L^{p'}(\dd\eta\otimes\dd(t,x,v))}^{-\frac{A}{B-A}}.
    \]
    A direct computation gives
    \[
        \frac{B}{B-A}
        =
        \frac{d\sigma+2d+\sigma}{2(d\sigma+d+\sigma)},
        \qquad
        -\frac{A}{B-A}
        =
        \frac{\sigma(d+1)}{2(d\sigma+d+\sigma)}.
    \]
    This proves the desired multiplicative estimate.
\end{proof}

\subsection{Nonlocal diffusion of Bessel type}

\begin{proof}[Proof of Theorem~\ref{thm:int:mainGN}]
    Choose $\alpha$ and $\beta$ with $0<\alpha<\beta<\alpha(1+1/\sigma)$. The admissible ratio $\beta/\alpha$
    will be fixed below.
    For solutions of
    \[
        (\partial_t+v\cdot\nabla_x)f=D_v^\sigma S
    \]
    Proposition~\ref{prop:representation} gives
    \[
        f-T_{K_\tau}f
        =
        T_{\int_0^\tau G_{r,\sigma} \dd r}(S)
        +
        T_{\int_0^\tau G^v_{r,\sigma} \dd r}(D_v^\sigma f).
    \]

    By Lemmas~\ref{lem:young-nl} and~\ref{lem:int-kernels},
    \[
        \norm{ T_{\int_0^\tau G^v_{r,\sigma} \dd r}(D_v^\sigma f)}_{\L^q}
        \lesssim \|D_v^\sigma f\|_{\L^p}
    \]
    whenever
    \[
        \frac1q+1=\frac1{\theta_1}+\frac1p,
        \qquad\text{i.e.}\qquad
        \frac1q+\frac{\sigma(\beta-\alpha)}{\homd}=\frac1p,
    \]
    and similarly
    \[
        \norm{ T_{\int_0^\tau G_{r,\sigma} \dd r}(S)}_{\L^q}
        \lesssim \|S\|_{\L^{p'}}
    \]
    whenever
    \[
        \frac1q+1=\frac1{\theta_2}+\frac1{p'},
        \qquad\text{i.e.}\qquad
        \frac1q+\frac{\alpha-\sigma(\beta-\alpha)}{\homd}=\frac1{p'}.
    \]
    We therefore seek $\beta/\alpha$ so that the two exponent relations give the same value of $q$.
    Writing $\rho:=\beta/\alpha$, this amounts to solving
    \[
        \frac1q+\frac{\sigma(\rho-1)}{(2\rho-1)d+1}=\frac1p,
        \qquad
        \frac1q+\frac{1-\sigma(\rho-1)}{(2\rho-1)d+1}=\frac1{p'}.
    \]
    A direct computation gives
    \begin{equation}
        \label{eq:ratio-beta-alpha}
        \frac\beta\alpha=\rho=
        \frac{d(p-2)+2p\sigma+2}{2(d(p-2)+p\sigma)},
    \end{equation}
    and
    \begin{equation}
        \label{eq:def-q-proof}
        q=\frac{2p(d\sigma+d+\sigma)}{d(p\sigma+2)}.
    \end{equation}
    To identify the admissible range of $p$, we inspect \eqref{eq:ratio-beta-alpha}.
    The denominator is $d(p-2)+p\sigma$, and hence is positive in the range of the theorem.
    Moreover,
    \[
        \frac\beta\alpha-1
        =\frac{2d+2-dp}{2(d(p-2)+p\sigma)},
    \]
    so $\beta>\alpha$ is equivalent to $p<2+\frac2d$.
    Likewise,
    \[
        1-\sigma\Bigl(\frac\beta\alpha-1\Bigr)
        =
        \frac{p(2d+d\sigma+2\sigma)-(4d+2d\sigma+2\sigma)}{2(d(p-2)+p\sigma)}.
    \]
    Thus, $\alpha-\sigma(\beta-\alpha)>0$ is equivalent to
    \[
        p>2-\frac{2\sigma}{2d+d\sigma+2\sigma}.
    \]
    This is the range stated in the theorem.

    With this choice, we have
    \[
        \|f-T_{K_\tau}f\|_{\L^q}
        \lesssim \|D_v^\sigma f\|_{\L^p}+\|S\|_{\L^{p'}}.
    \]

    We next estimate the mollified term.
    Let $\theta\in[1,\infty]$ be defined by
    \[
        \frac1q+1=\frac1\theta+\frac1p.
    \]
    Lemmas~\ref{lem:young-nl} and~\ref{lem:kernel-Kr} yield
    \[
        \|T_{K_\tau}f\|_{\L^q}
        \lesssim \|K_\tau\|_{\L^\theta}\,\|f\|_{\L^p}
        \lesssim \tau^{\homd(\frac1\theta-1)}\|f\|_{\L^p}
        =\tau^{\homd(\frac1q-\frac1p)}\|f\|_{\L^p}.
    \]
    The equivalence
    \[
        q>p
        \iff
        p<2+\frac2d,
    \]
    shows that $q>p$ in the admissible range.  Hence the exponent of $\tau$ is negative; therefore
    \[
        \|T_{K_\tau}f\|_{\L^q}\to0
        \qquad\text{as }\tau\to\infty.
    \]
    Consequently,
    \begin{equation}
        \label{eq:linear-estimate}
        \|f\|_{\L^q}
        \lesssim \|D_v^\sigma f\|_{\L^p}+\|S\|_{\L^{p'}}.
    \end{equation}

    It remains to balance the two terms in \eqref{eq:linear-estimate}.
    Again, if $S = 0$, we have $f = 0$ and there is nothing to prove.
    We may thus assume $S \neq 0$ and $f \neq 0$.
    For $\nu>0$, define
    \[
        f_\nu(t,x,v):=f(\nu t,\nu x,v),
        \qquad
        S_\nu(t,x,v):=\nu S(\nu t,\nu x,v).
    \]
    Then $(\partial_t+v\cdot\nabla_x)f_\nu=D_v^\sigma S_\nu$.
    Applying \eqref{eq:linear-estimate} to $(f_\nu,S_\nu)$ and rescaling yields
    \[
        \|f\|_{\L^q}
        \lesssim \nu^{(d+1)(\frac1q-\frac1p)}\|D_v^\sigma f\|_{\L^p}
        +
        \nu^{1+(d+1)(\frac1q-\frac1{p'})}\|S\|_{\L^{p'}}.
    \]
    Set
    \[
        A:=(d+1)\Bigl(\frac1q-\frac1p\Bigr),
        \qquad
        B:=1+(d+1)\Bigl(\frac1q-\frac1{p'}\Bigr).
    \]
    By \eqref{eq:def-q-proof}, we have
    \[
        A=-\frac{\sigma(d+1)(2d+2-dp)}{2p(d\sigma+d+\sigma)}<0,
        \qquad
        B=\frac{(2d+2-dp)(d\sigma+2d+\sigma)}{2p(d\sigma+d+\sigma)}>0.
    \]
    Hence
    \[
        \|f\|_{\L^q}\lesssim \nu^A\|D_v^\sigma f\|_{\L^p}+\nu^B\|S\|_{\L^{p'}},
        \qquad \nu>0.
    \]
    Optimising in $\nu$ gives
    \[
        \|f\|_{\L^q}
        \lesssim
        \|D_v^\sigma f\|_{\L^p}^{\frac{B}{B-A}}
        \|S\|_{\L^{p'}}^{-\frac{A}{B-A}}.
    \]
    A direct computation yields
    \[
        \frac{B}{B-A}=\frac{d\sigma+2d+\sigma}{2(d\sigma+d+\sigma)},
        \qquad
        -\frac{A}{B-A}=\frac{\sigma(d+1)}{2(d\sigma+d+\sigma)},
    \]
    which is the desired estimate.
    This proves the theorem.
\end{proof}

\appendix
\section{Decay estimates for fractional derivatives and integrals}
We recall two standard consequences of the kernel representations of
the fractional Laplacian and of Riesz potentials; see, for instance,
\cite[Section~3]{MR2944369} and
\cite[Section~1.2.1]{MR3243741}. We include proofs for the reader's convenience.

\begin{lemma}
    \label{lem:family-fractional}
    Let $0<\sigma<1$ and $R>0$.

    \begin{enumerate}[itemsep=0.2cm]
        \item[(i)] If $\mathscr B\subset \C_c^\infty(B_R(0))$ is bounded in $\C^1$, then there exists
              $C=C(d,\sigma,R,\mathscr B)>0$ such that
              \[
                  \abs{[D_v^\sigma \varphi](w)}\le C(1+\abs{w})^{-d-\sigma}
                  \qquad\text{for all }\varphi\in\mathscr B,\ w\in\R^d.
              \]
        \item[(ii)] If, in addition, every $\varphi\in\mathscr B$ has zero average, then there exists a
              constant $C=C(d,\sigma,R,\mathscr B)>0$ such that
              \[
                  \abs{[I_v^\sigma \varphi](w)}\le C(1+\abs{w})^{-d-1+\sigma}
                  \qquad\text{for all }\varphi\in\mathscr B,\ w\in\R^d.
              \]
    \end{enumerate}
\end{lemma}

\begin{proof}
    For (i), we use the singular integral formula
    \[
        [D_v^\sigma \varphi](w)=c_{d,\sigma}^{\mathrm{Bes}}\,\PV\int_{\R^d}
        \frac{\varphi(w)-\varphi(w+h)}{\abs{h}^{d+\sigma}}\dd h.
    \]
    If $\abs{w}\le 2R$, then the integrand is controlled by the $\C^1$ norm of $\varphi$ for small $h$ and by
    the $\L^\infty$ norm for large $h$, hence $\abs{D_v^\sigma\varphi(w)}\le C$.
    If $\abs{w}>2R$, then $\varphi(w)=0$ and
    \[
        \abs{[D_v^\sigma\varphi](w)}
        \le c_{d,\sigma}^{\mathrm{Bes}}\,\int_{B_R(0)} \frac{\abs{\varphi(z)}}{\abs{w-z}^{d+\sigma}}\dd z
        \lesssim \abs{w}^{-d-\sigma}.
    \]
    This proves (i).

    For (ii), we write
    \[
        [I_v^\sigma\varphi](w)=c_{d,\sigma}^{\mathrm{Rie}}\,\int_{\R^d}\abs{w-z}^{\sigma-d}\varphi(z)\dd z.
    \]
    If $\abs{w}\le 2R$, the local integrability of $\abs{w-z}^{\sigma-d}$ and the uniform $\L^\infty$ bound on
    $\varphi$ imply $\abs{[I_v^\sigma\varphi](w)}\le C$.
    If $\abs{w}>2R$, the zero-average condition gives
    \[
        [I_v^\sigma\varphi](w)=c_{d,\sigma}^{\mathrm{Rie}}\,
        \int_{B_R(0)}
        \bigl(\abs{w-z}^{\sigma-d}-\abs{w}^{\sigma-d}\bigr)\varphi(z)\dd z.
    \]
    Since $z\in B_R(0)$ and $\abs{w}>2R$, the mean value theorem yields
    \[
        \abs{\abs{w-z}^{\sigma-d}-\abs{w}^{\sigma-d}}\lesssim \abs{z}\,\abs{w}^{\sigma-d-1}\lesssim
        \abs{w}^{\sigma-d-1}.
    \]
    Therefore $\abs{[I_v^\sigma\varphi](w)}\lesssim \abs{w}^{-d-1+\sigma}$.
\end{proof}

The following estimate is standard. It follows from the kernel representations
of Riesz potentials and Riesz transforms; see, for instance,
\cite[Ch.~V, Sec.~1 and Ch.~III, Sec.~1]{MR290095}
or \cite[Sec.~1.2.1]{MR3243741}.

\begin{lemma} \label{lem:potential-divergence-decay}
    Let $d\ge 1$, $0<\sigma<1$, and let $N>d+2$ be an integer.
    Let $R\ge1$, and suppose
    \[
        \Theta\in \C_c^N(\R^d;\R^d),
        \qquad
        \supp\Theta\subset B_R,
    \]
    with
    \[
        M_N(\Theta):=
        \sum_{j=1}^d\sum_{\abs{\alpha}\le N}
        \|\partial^\alpha \Theta_j\|_{\L^\infty(\R^d)}
        <\infty .
    \]
    Define
    \[
        U=(-\Delta)^{-\sigma}\nabla\cdot\Theta
    \]
    by the Fourier multiplier formula
    \[
        \widehat U(\xi)
        =
        \abs{\xi}^{-2\sigma} i\xi\cdot \widehat\Theta(\xi).
    \]
    Then $U\in \C^1(\R^d)$, and there is a constant
    $C=C(d,\sigma,N,R)$ such that, for every $w\in\R^d$,
    \[
        \abs{U(w)}
        \le
        C M_N(\Theta)(1+\abs{w})^{-d-1+2\sigma},
    \]
    and
    \[
        \abs{\nabla U(w)}
        \le
        C M_N(\Theta)(1+\abs{w})^{-d-2+2\sigma}.
    \]
    In particular, the estimates are uniform for any family of vector fields
    $\Theta$ supported in $B_R$ and bounded in $\C^N$.
\end{lemma}

\begin{proof}
    We split the proof into a local estimate and a far-field estimate.

    First, since $\Theta$ is supported in $B_R$, integration by parts gives the
    standard Fourier decay bound
    \[
        \abs{\widehat\Theta(\xi)}
        \le
        C_{d,N,R} M_N(\Theta)(1+\abs{\xi})^{-N}.
    \]
    Indeed, for $\abs{\xi}\le1$ this follows from
    $\|\Theta\|_{\L^1}\le C_{d,R}M_N(\Theta)$, while for $\abs{\xi}>1$ one
    integrates by parts $N$ times in a coordinate direction satisfying
    $\abs{\xi_k}\ge \abs{\xi}/\sqrt d$.

    Therefore
    \[
        \abs{U(w)}
        \le
        C M_N(\Theta)
        \int_{\R^d}
        \abs{\xi}^{1-2\sigma}(1+\abs{\xi})^{-N}\dd \xi .
    \]
    The integral is finite. Near $\xi=0$, its radial exponent is
    \[
        d-1+1-2\sigma=d-2\sigma>-1,
    \]
    because $d\ge1$ and $0<\sigma<1$. At infinity it is finite because
    $N>d+2$. Hence
    \[
        \|U\|_{\L^\infty}
        \le
        C M_N(\Theta).
    \]

    Similarly,
    \[
        \widehat{\partial_\ell U}(\xi)
        =
        i\xi_\ell \abs{\xi}^{-2\sigma} i\xi\cdot\widehat\Theta(\xi),
    \]
    so
    \[
        \abs{\nabla U(w)}
        \le
        C M_N(\Theta)
        \int_{\R^d}
        \abs{\xi}^{2-2\sigma}(1+\abs{\xi})^{-N}\dd \xi .
    \]
    This integral is also finite: near zero the radial exponent is
    \[
        d-1+2-2\sigma=d+1-2\sigma>-1,
    \]
    and at infinity we again use $N>d+2$. Thus
    \[
        \|\nabla U\|_{\L^\infty}
        \le
        C M_N(\Theta).
    \]
    This proves the desired estimates on bounded sets of $w$, after possibly
    increasing the constant.

    It remains to prove the decay for large $\abs{w}$. Let
    \[
        K_j=\cF^{-1}\bigl(i\xi_j\abs{\xi}^{-2\sigma}\bigr).
    \]
    The multiplier $i\xi_j\abs{\xi}^{-2\sigma}$ is smooth away from the origin and
    homogeneous of degree $1-2\sigma$. Hence $K_j$ is smooth away from the
    origin and homogeneous of degree $-d-(1-2\sigma)=-d-1+2\sigma$.
    Equivalently, away from the origin one has the explicit form
    \[
        K_j(z)=c_{d,\sigma}\frac{z_j}{\abs{z}^{d+2-2\sigma}},
    \]
    with the usual logarithmic interpretation in the borderline potential case.
    Consequently,
    \[
        \abs{K_j(z)}
        \le
        C\abs{z}^{-d-1+2\sigma},
        \qquad z\ne0,
    \]
    and
    \[
        \abs{\nabla K_j(z)}
        \le
        C\abs{z}^{-d-2+2\sigma},
        \qquad z\ne0.
    \]

    Since
    \[
        \widehat U(\xi)
        =
        \sum_{j=1}^d i\xi_j\abs{\xi}^{-2\sigma}\widehat\Theta_j(\xi),
    \]
    we have, in the sense of distributions,
    \[
        U=\sum_{j=1}^d K_j*\Theta_j.
    \]
    If $\abs{w}\ge 2R$, then for every $z\in\supp\Theta\subset B_R$,
    $\abs{w-z}\ge \abs{w}-\abs{z}\ge \frac{\abs{w}}{2}$.
    Therefore the convolution is an absolutely convergent ordinary integral for
    such $w$, and
    \[
        \begin{aligned}
            \abs{U(w)}
             & \le
            \sum_{j=1}^d
            \int_{B_R}\abs{K_j(w-z)}\,\abs{\Theta_j(z)}\dd z \\
             & \le
            C M_N(\Theta)
            \int_{B_R}\abs{w-z}^{-d-1+2\sigma}\dd z          \\
             & \le
            C M_N(\Theta)\abs{w}^{-d-1+2\sigma}.
        \end{aligned}
    \]
    Likewise,
    \[
        \begin{aligned}
            \abs{\nabla U(w)}
             & \le
            \sum_{j=1}^d
            \int_{B_R}\abs{\nabla K_j(w-z)}\,\abs{\Theta_j(z)}\dd z \\
             & \le
            C M_N(\Theta)
            \int_{B_R}\abs{w-z}^{-d-2+2\sigma}\dd z                 \\
             & \le
            C M_N(\Theta)\abs{w}^{-d-2+2\sigma}.
        \end{aligned}
    \]

    Since $R\ge1$, the powers of $\abs{w}$ are comparable to the corresponding
    powers of $1+\abs{w}$ on the region $\abs{w}\ge2R$. On the complementary region
    $\abs{w}\le2R$, the already proved uniform bounds imply
    \[
        \abs{U(w)}
        \le
        C M_N(\Theta)
        \le
        C M_N(\Theta)(1+\abs{w})^{-d-1+2\sigma},
    \]
    after increasing $C$. The same argument gives
    \[
        \abs{\nabla U(w)}
        \le
        C M_N(\Theta)(1+\abs{w})^{-d-2+2\sigma}.
    \]
    Combining the local and far-field estimates proves the lemma.
\end{proof}

\section{Critical integration}
To estimate the integrated kernels, we first explain how to carry out the critical integration by real
interpolation. This is a useful technical tool, which may be of independent interest. We believe it should be
known, but we could not find a reference, so we provide a proof for the convenience of the reader.

\begin{lemma}
    \label{lem:critical-integration-interpolation}
    Let $N\in\N$, and let $J:(0,\infty) \times \R^N \to \R$, $(r,x) \mapsto J_r(x)$, be jointly measurable
    with $J_r\in \L^1(\R^N)\cap \L^\infty(\R^N)$ for
    every $r>0$. Assume that for some $A>0$ and $\mu,\nu>0$,
    \[
        \|J_r\|_{\L^1(\R^N)}
        \le A\,r^{-1+\mu},
        \qquad
        \|J_r\|_{\L^\infty(\R^N)}
        \le A\,r^{-1-\nu}
        \qquad\text{for all }r>0.
    \]
    Then, uniformly in $\tau>0$,
    \[
        \left\|\int_0^\tau J_r\dd r\right\|_{\L^{\theta,\infty}(\R^N)}
        \lesssim A,
        \qquad
        \theta:=\frac{\mu+\nu}{\nu}.
    \]
\end{lemma}

\begin{proof}
    Set
    \[
        F_\tau(x):=\int_0^\tau J_r(x)\dd r.
    \]
    Since $\mu>0$, the integral is absolutely convergent in $\L^1$.

    Let $K(s,\cdot;\L^1,\L^\infty)$ denote the real interpolation
    $K$-functional; see \cite{MR482275}. For every $s>0$ and every $\delta\in(0,\tau]$,
    we split
    \[
        F_\tau
        =
        \int_0^\delta J_r\dd r
        +
        \int_\delta^\tau J_r\dd r,
    \]
    and therefore
    \begin{align*}
        K(s,F_\tau;\L^1,\L^\infty)
         & \le
        \left\|\int_0^\delta J_r\dd r\right\|_{\L^1}
        +
        s\left\|\int_\delta^\tau J_r\dd r\right\|_{\L^\infty}
        \\
         & \le
        \int_0^\delta \|J_r\|_{\L^1}\dd r
        +
        s\int_\delta^\tau \|J_r\|_{\L^\infty}\dd r
        \\
         & \le
        A\bigl(\delta^\mu+s\,\delta^{-\nu}\bigr).
    \end{align*}
    Choose
    \[
        \delta:=\min\{\tau,s^{1/(\mu+\nu)}\}.
    \]
    If $\delta=s^{1/(\mu+\nu)}$, then
    \[
        K(s,F_\tau;\L^1,\L^\infty)
        \lesssim A\,s^\lambda,
        \qquad
        \lambda:=\frac{\mu}{\mu+\nu}.
    \]
    If $\delta=\tau$, then the second integral vanishes and $s\ge \tau^{\mu+\nu}$, hence
    \[
        K(s,F_\tau;\L^1,\L^\infty)
        \lesssim A\,\tau^\mu
        \le A\,s^\lambda.
    \]
    Thus
    \[
        K(s,F_\tau;\L^1,\L^\infty)\lesssim A\,s^\lambda
        \qquad\text{for all }s>0.
    \]
    This means that
    \[
        F_\tau\in (\L^1,\L^\infty)_{\lambda,\infty}
        \qquad\text{and}\qquad
        \|F_\tau\|_{(\L^1,\L^\infty)_{\lambda,\infty}}
        \lesssim A.
    \]
    By the standard real interpolation identity,
    \[
        (\L^1,\L^\infty)_{\lambda,\infty}
        =
        \L^{\theta,\infty},
        \qquad
        \frac1\theta=1-\lambda=\frac{\nu}{\mu+\nu}.
    \]
    This proves the lemma.
\end{proof}

\bibliographystyle{plain}

\begin{thebibliography}{10}

    \bibitem{anceschi2025poincareinequalityquantitativegiorgi}
    Francesca Anceschi, Helge Dietert, Jessica Guerand, Amélie Loher, Clément Mouhot, and Annalaura Rebucci.
    \newblock Poincar\'e inequality and quantitative {De Giorgi} method for hypoelliptic operators, 2024.
    \newblock arXiv:2401.12194.

    \bibitem{MR4950599}
    Francesca Anceschi and Mirco Piccinini.
    \newblock Boundedness estimates for nonlinear nonlocal kinetic {K}olmogorov-{F}okker-{P}lanck equations.
    \newblock {\em NoDEA Nonlinear Differential Equations Appl.}, 32(6):Paper No. 121, 25, 2025.

    \bibitem{MR4898687}
    Pascal Auscher, Cyril Imbert, and Lukas Niebel.
    \newblock Fundamental solutions to {K}olmogorov-{F}okker-{P}lanck equations with rough coefficients: existence, uniqueness, upper estimates.
    \newblock {\em SIAM J. Math. Anal.}, 57(2):2114--2137, 2025.

    \bibitem{auscher2025weaksolutionskolmogorovfokkerplanckequations}
    Pascal Auscher, Cyril Imbert, and Lukas Niebel.
    \newblock Weak solutions to {K}olmogorov-{F}okker-{P}lanck equations: regularity, existence and uniqueness, 2025.
    \newblock arXiv:2403.17464.

    \bibitem{auscher2026kineticsobolevspaces}
    Pascal Auscher and Lukas Niebel.
    \newblock Kinetic {S}obolev {S}paces, 2026.
    \newblock arXiv:2603.17491.

    \bibitem{MR482275}
    J\"oran Bergh and J\"orgen L\"ofstr\"om.
    \newblock {\em Interpolation spaces. {A}n introduction}, volume No. 223 of {\em Grundlehren der Mathematischen Wissenschaften}.
    \newblock Springer-Verlag, Berlin-New York, 1976.

    \bibitem{MR4964333}
    Anup Biswas and Erwin Topp.
    \newblock Lipschitz regularity of fractional {$p$}-{L}aplacian.
    \newblock {\em Ann. PDE}, 11(2):Paper No. 27, 43, 2025.

    \bibitem{MR1949176}
    F.~Bouchut.
    \newblock Hypoelliptic regularity in kinetic equations.
    \newblock {\em J. Math. Pures Appl. (9)}, 81(11):1135--1159, 2002.

    \bibitem{MR4942309}
    Sun-Sig Byun and Kyeongbae Kim.
    \newblock {$L^q$} estimates for nonlocal {$p$}-{L}aplacian-type equations with {BMO} kernel coefficients in divergence form.
    \newblock {\em Commun. Contemp. Math.}, 27(9):Paper No. 2550012, 78, 2025.

    \bibitem{MR4303657}
    F\'elix del Teso, David G\'omez-Castro, and Juan~Luis V\'azquez.
    \newblock Three representations of the fractional {$p$}-{L}aplacian: semigroup, extension and {B}alakrishnan formulas.
    \newblock {\em Fract. Calc. Appl. Anal.}, 24(4):966--1002, 2021.

    \bibitem{MR3237774}
    Agnese Di~Castro, Tuomo Kuusi, and Giampiero Palatucci.
    \newblock Nonlocal {H}arnack inequalities.
    \newblock {\em J. Funct. Anal.}, 267(6):1807--1836, 2014.

    \bibitem{MR3542614}
    Agnese Di~Castro, Tuomo Kuusi, and Giampiero Palatucci.
    \newblock Local behavior of fractional {$p$}-minimizers.
    \newblock {\em Ann. Inst. H. Poincar\'e{} C Anal. Non Lin\'eaire}, 33(5):1279--1299, 2016.

    \bibitem{MR2944369}
    Eleonora Di~Nezza, Giampiero Palatucci, and Enrico Valdinoci.
    \newblock Hitchhiker's guide to the fractional {S}obolev spaces.
    \newblock {\em Bull. Sci. Math.}, 136(5):521--573, 2012.

    \bibitem{dietert2025criticaltrajectorieskineticgeometry}
    Helge Dietert, Clément Mouhot, Lukas Niebel, and Rico Zacher.
    \newblock Critical trajectories in kinetic geometry, 2025.
    \newblock arXiv:2508.14868.

    \bibitem{dietert2025nashsgboundkolmogorov}
    Helge Dietert and Lukas Niebel.
    \newblock Nash's ${G}$ bound for the {K}olmogorov equation, 2025.
    \newblock arXiv:2510.21621.

    \bibitem{dnz_kinpLaplace_2026}
    Helge Dietert, Lukas Niebel, and Rico Zacher.
    \newblock Nonlinear kinetic diffusion equations with $p$-growth, 2026.
    \newblock arXiv:2605.18521.

    \bibitem{MR4204564}
    Mengyao Ding, Chao Zhang, and Shulin Zhou.
    \newblock Local boundedness and {H}\"older continuity for the parabolic fractional {$p$}-{L}aplace equations.
    \newblock {\em Calc. Var. Partial Differential Equations}, 60(1):Paper No. 38, 45, 2021.

    \bibitem{MR3023366}
    Qiang Du, Max Gunzburger, Richard~B. Lehoucq, and Kun Zhou.
    \newblock Analysis and approximation of nonlocal diffusion problems with volume constraints.
    \newblock {\em SIAM Rev.}, 54(4):667--696, 2012.

    \bibitem{MR4468371}
    Prashanta Garain and Kaj Nystr\"om.
    \newblock On regularity and existence of weak solutions to nonlinear {K}olmogorov-{F}okker-{P}lanck type equations with rough coefficients.
    \newblock {\em Math. Eng.}, 5(2):Paper No. 043, 37, 2023.

    \bibitem{giovagnoli2025c1alpharegularityfractionalpharmonic}
    Davide Giovagnoli, David Jesus, and Luis Silvestre.
    \newblock ${C}^{1+\alpha}$ regularity for fractional $p$-harmonic functions, 2025.
    \newblock arXiv:2509.26565.

    \bibitem{MR3923847}
    Fran\c{c}ois Golse, Cyril Imbert, Cl\'ement Mouhot, and Alexis~F. Vasseur.
    \newblock Harnack inequality for kinetic {F}okker-{P}lanck equations with rough coefficients and application to the {L}andau equation.
    \newblock {\em Ann. Sc. Norm. Super. Pisa Cl. Sci. (5)}, 19(1):253--295, 2019.

    \bibitem{MR2445437}
    Loukas Grafakos.
    \newblock {\em Classical {F}ourier analysis}, volume 249 of {\em Graduate Texts in Mathematics}.
    \newblock Springer, New York, second edition, 2008.

    \bibitem{MR3243741}
    Loukas Grafakos.
    \newblock {\em Modern {F}ourier analysis}, volume 250 of {\em Graduate Texts in Mathematics}.
    \newblock Springer, New York, third edition, 2014.

    \bibitem{MR2728700}
    Max Gunzburger and Richard~B. Lehoucq.
    \newblock A nonlocal vector calculus with application to nonlocal boundary value problems.
    \newblock {\em Multiscale Model. Simul.}, 8(5):1581--1598, 2010.

    \bibitem{MR4711580}
    Solveig Hepp and Moritz Kassmann.
    \newblock The divergence theorem and nonlocal counterparts.
    \newblock {\em Bull. Lond. Math. Soc.}, 56(2):711--733, 2024.

    \bibitem{MR3593528}
    Antonio Iannizzotto, Sunra Mosconi, and Marco Squassina.
    \newblock Global {H}\"older regularity for the fractional {$p$}-{L}aplacian.
    \newblock {\em Rev. Mat. Iberoam.}, 32(4):1353--1392, 2016.

    \bibitem{MR4049224}
    Cyril Imbert and Luis Silvestre.
    \newblock The weak {H}arnack inequality for the {B}oltzmann equation without cut-off.
    \newblock {\em J. Eur. Math. Soc. (JEMS)}, 22(2):507--592, 2020.

    \bibitem{MR454618}
    Bj\"orn Jawerth.
    \newblock Some observations on {B}esov and {L}izorkin-{T}riebel spaces.
    \newblock {\em Math. Scand.}, 40(1):94--104, 1977.

    \bibitem{MR4828346}
    Moritz Kassmann and Marvin Weidner.
    \newblock The {H}arnack inequality fails for nonlocal kinetic equations.
    \newblock {\em Adv. Math.}, 459:Paper No. 110030, 14, 2024.

    \bibitem{MR4030247}
    Janne Korvenp\"a\"a, Tuomo Kuusi, and Erik Lindgren.
    \newblock Equivalence of solutions to fractional {$p$}-{L}aplace type equations.
    \newblock {\em J. Math. Pures Appl. (9)}, 132:1--26, 2019.

    \bibitem{MR4679981}
    Naian Liao.
    \newblock H\"older regularity for parabolic fractional {$p$}-{L}aplacian.
    \newblock {\em Calc. Var. Partial Differential Equations}, 63(1):Paper No. 22, 34, 2024.

    \bibitem{MR4688651}
    Am\'elie Loher.
    \newblock Quantitative {D}e {G}iorgi methods in kinetic theory for non-local operators.
    \newblock {\em J. Funct. Anal.}, 286(6):Paper No. 110312, 67, 2024.

    \bibitem{MR3491533}
    Jos\'e{}~M. Maz\'on, Julio~D. Rossi, and Juli\'an Toledo.
    \newblock Fractional {$p$}-{L}aplacian evolution equations.
    \newblock {\em J. Math. Pures Appl. (9)}, 105(6):810--844, 2016.

    \bibitem{n_kinetictor_2026}
    Lukas Niebel.
    \newblock A new proof of the transfer of regularity for kinetic equations, 2026.
    \newblock arXiv:2605.13582.

    \bibitem{MR4350284}
    Lukas Niebel and Rico Zacher.
    \newblock Kinetic maximal {$L^2$}-regularity for the (fractional) {K}olmogorov equation.
    \newblock {\em J. Evol. Equ.}, 21(3):3585--3612, 2021.

    \bibitem{MR4336467}
    Lukas Niebel and Rico Zacher.
    \newblock Kinetic maximal {$L^p$}-regularity with temporal weights and application to quasilinear kinetic diffusion equations.
    \newblock {\em J. Differential Equations}, 307:29--82, 2022.

    \bibitem{MR4875497}
    Lukas Niebel and Rico Zacher.
    \newblock On a kinetic {P}oincar\'e{} inequality and beyond.
    \newblock {\em J. Funct. Anal.}, 289(1):Paper No. 110899, 18, 2025.

    \bibitem{MR2068847}
    Andrea Pascucci and Sergio Polidoro.
    \newblock The {M}oser's iterative method for a class of ultraparabolic equations.
    \newblock {\em Commun. Contemp. Math.}, 6(3):395--417, 2004.

    \bibitem{MR3714833}
    Armin Schikorra, Tien-Tsan Shieh, and Daniel~E. Spector.
    \newblock Regularity for a fractional {$p$}-{L}aplace equation.
    \newblock {\em Commun. Contemp. Math.}, 20(1):1750003, 6, 2018.

    \bibitem{MR3403430}
    Tien-Tsan Shieh and Daniel~E. Spector.
    \newblock On a new class of fractional partial differential equations.
    \newblock {\em Adv. Calc. Var.}, 8(4):321--336, 2015.

    \bibitem{MR4807233}
    Luis Silvestre.
    \newblock Regularity estimates and open problems in kinetic equations.
    \newblock In {\em {{$\rm A^3N^2M$}}: approximation, applications, and analysis of nonlocal, nonlinear models}, volume 165 of {\em IMA Vol. Math. Appl.}, pages 101--148. Springer, Cham, [2023] \copyright 2023.

    \bibitem{MR290095}
    Elias~M. Stein.
    \newblock {\em Singular integrals and differentiability properties of functions}, volume No. 30 of {\em Princeton Mathematical Series}.
    \newblock Princeton University Press, Princeton, NJ, 1970.

    \bibitem{MR4039522}
    Logan~F. Stokols.
    \newblock H\"older continuity for a family of nonlocal hypoelliptic kinetic equations.
    \newblock {\em SIAM J. Math. Anal.}, 51(6):4815--4847, 2019.

\end{thebibliography}

\end{document}